\documentclass{scrartcl}
\usepackage{latexsym} 
\usepackage{amsmath}
\usepackage{amsthm}
\usepackage{amsfonts}
\usepackage{amssymb}
\usepackage{amsxtra}
\usepackage{amscd}
\usepackage[shortlabels]{enumitem}
\usepackage{mathrsfs}
\usepackage{hyperref}
\usepackage{mathtools}
\usepackage{bbm}

\hypersetup{
breaklinks=true,
colorlinks=true,
linkcolor=blue,
citecolor=blue,
urlcolor=blue,
filecolor=blue,
}

\numberwithin{equation}{section} 
\setkomafont{title}{\normalfont}

\newenvironment{pdeq}{ \left\{ \begin{aligned}}{\end{aligned}\right.}

\newcommand{\eqrefsub}[2]{\ensuremath{\eqref{#1}_{#2}}}
\newcommand{\np}[1]{(#1)}
\newcommand{\nb}[1]{[#1]}
\newcommand{\bp}[1]{\big(#1\big)}
\newcommand{\bb}[1]{\big[#1\big]}
\newcommand{\Bp}[1]{\bigg(#1\bigg)}

\newcommand{\Bb}[1]{\bigg[#1\bigg]}
\newcommand{\calb}{{\mathcal B}}

\newcommand{\calp}{{\mathcal P}}

\newcommand{\calt}{{\mathcal T}}

\newcommand{\R}{\mathbb{R}}
\newcommand{\Q}{\mathbb{Q}}
\newcommand{\C}{\mathbb{C}}
\newcommand{\Z}{\mathbb{Z}}
\newcommand{\N}{\mathbb{N}}
\DeclareMathOperator{\e}{e}

\DeclareMathOperator{\Div}{div}

\DeclareMathOperator{\supp}{supp}

\DeclareMathOperator{\dist}{dist}

\newcommand{\embeds}{\hookrightarrow}

\newcommand{\ra}{\rightarrow}

\newcommand{\wto}{\rightharpoonup}
\newcommand{\set}[1]{\ensuremath{\{#1\}}}

\newcommand{\setc}[2]{\ensuremath{\{#1 : #2\}}}
\newcommand{\setcl}[2]{\ensuremath{\bigl\{#1 : #2\bigr\}}}
\newcommand{\setcL}[2]{\ensuremath{\biggl\{#1\, :\, #2\biggr\}}}

\newcommand{\ball}{B}

\renewcommand{\restriction}[2]{#1\big | _{#2}}
\newcommand{\proj}{\calp}

\newcommand{\projhOm}{\proj_\Omega}
\newcommand{\projhm}{\proj_{\Omega_m}}
\newcommand{\grp}{G}
\newcommand{\dualgrp}{\widehat{G}}

\newcommand{\grph}{H}

\newcommand{\torus}{{\mathbb T}}

\newcommand{\transpose}{\top}

\newcommand{\idmatrix}{I}
\newcommand{\grad}{\nabla}

\newcommand{\pt}{\partial_t}

\newcommand{\dx}{{\mathrm d}x}

\newcommand{\dt}{{\mathrm d}t}

\newcommand{\dxi}{{\mathrm d}\xi}

\newcommand{\SR}{\mathscr{S}}

\newcommand{\TDR}{\mathscr{S^\prime}}

\newcommand{\FT}{\mathscr{F}}
\newcommand{\iFT}{\mathscr{F}^{-1}}

\newcommand{\AR}{\mathrm{A}}
\newcommand{\norm}[1]{\lVert#1\rVert}

\newcommand{\norml}[1]{\bigl\lVert#1\bigr\rVert}

\newcommand{\snorm}[1]{{\lvert #1 \rvert}}
\newcommand{\snorml}[1]{{\bigl\lvert #1 \big\rvert}}
\newcommand{\snormL}[1]{{\Bigl\lvert #1 \Big\rvert}}

\newcommand{\WSR}[2]{\mathrm{W}^{#1,#2}} 
\newcommand{\WSRN}[2]{\mathrm{W}^{#1,#2}_0}

\newcommand{\WSRloc}[2]{\mathrm{W}^{#1,#2}_{\mathrm{loc}}} 
\newcommand{\CR}[1]{\mathrm{C}^{#1}}  
\newcommand{\LR}[1]{\mathrm{L}^{#1}}
 
\newcommand{\lR}[1]{\ell^{#1}}

\newcommand{\LRloc}[1]{\mathrm{L}^{#1}_{\mathrm{loc}}} 
\newcommand{\CRi}{\CR \infty}
\newcommand{\CRci}{\CR \infty_0}

\newcommand{\LRsigma}[1]{\mathrm{L}^{#1}_{\sigma}}

\newcommand{\XT}{{\mathcal X_{\tay}^q}}
\newcommand{\Xs}{{\mathrm X_{\tay,s}^q}}
\newcommand{\Xk}{{\mathrm X_{\tay,\perf k}^q}}
\newcommand{\XreyT}{{\mathcal X_{\rey,\tay}^q}}

\newcommand{\YT}{\mathcal Y^q}
\newcommand{\Ys}{\mathrm Y^q}

\newcommand{\nsnonlinb}[2]{#1\cdot\grad #2}

\newcommand{\vvel}{v}
\newcommand{\vpres}{p}

\newcommand{\wvel}{w}
\newcommand{\wpres}{\mathfrak{q}}

\newcommand{\twvel}{\widetilde{w}}

\newcommand{\twpres}{\widetilde{\pi}}

\newcommand{\uvel}{u}

\newcommand{\upres}{\mathfrak{p}}

\newcommand{\tuvel}{\widetilde{u}}
\newcommand{\tupres}{\widetilde{\mathfrak{p}}}

\newcommand{\rotterm}[1]{\np{ \eone\wedge #1 - \nsnonlinb{\eone\wedge x}{#1}}}

\newcommand{\rottermsimple}[1]{ \eone\wedge #1 - \nsnonlinb{\eone\wedge x}{#1} }

\newcommand{\tin}{\text{in }}
\newcommand{\tif}{\text{if }}
\newcommand{\ton}{\text{on }}

\newcommand{\tfor}{\text{for }}

\renewcommand{\epsilon}{\varepsilon}
\renewcommand{\phi}{\varphi}
\newcommand{\rey}{\lambda}

\newcommand{\tay}{\omega}

\newcommand{\per}{\calt}

\newcommand{\perf}{\frac{2\pi}{\per}}
\newcommand{\perfs}{\tfrac{2\pi}{\per}}

\newcommand{\eone}{\e_1}

\newcommand{\rotmatrix}{Q}

\newcommand{\cutoff}{\chi}

\newcommand{\change}[1]{}

\theoremstyle{plain}
\newtheorem{thm}{Theorem}[section]

\newtheorem{lem}[thm]{Lemma}

\theoremstyle{remark}
\newtheorem{rem}[thm]{Remark}

\begin{document}
\title{On the Stokes-type resolvent problem
associated with
time-periodic flow
around a rotating obstacle}

\author{Thomas Eiter%
\footnote{%
Weierstrass Institute for Applied Analysis and Stochastics,
Mohrenstra\ss{}e 39, 10117 Berlin, Germany.
Email: {\texttt{thomas.eiter@wias-berlin.de}}} 
}

\maketitle

\begin{abstract}
Consider the resolvent problem associated with the linearized viscous flow around a rotating body.
Within a setting of classical Sobolev spaces,
this problem is not well posed on the whole imaginary axis.
Therefore, a framework of homogeneous Sobolev spaces is introduced
where existence of a unique solution can be guaranteed for every purely imaginary resolvent parameter.
For this purpose, the problem is reduced to an auxiliary problem,
which is studied by means of Fourier analytic tools in a group setting.
In the end, uniform resolvent estimates can be derived, 
which lead to the existence of solutions to the associated time-periodic linear problem. 
\end{abstract}


\section{Introduction}

The present article is mainly concerned with the study of 
the problem
\begin{equation}\label{sys:Stokes.rot.res}
\begin{pdeq}
is\vvel+\tay\rotterm{\vvel} 
- \Delta \vvel
+ \grad \vpres
 &= g
&& \tin\Omega, \\
\Div\vvel&=0
&& \tin\Omega, \\
\vvel&=0
&& \ton \partial\Omega
\end{pdeq}
\end{equation}
in a three-dimensional exterior domain $\Omega\subset\R^3$.
Here
$s\in\R$ and $\tay>0$ are given parameters, 
$g\colon\Omega\to\R^3$ is a given vector field,
and $\vvel\colon\Omega\to\R^3$ and $\vpres\colon\Omega\to\R$ are 
the unknown functions.
Then \eqref{sys:Stokes.rot.res} can be regarded as a resolvent problem
with a purely imaginary resolvent parameter $is$, $s\in\R$.
Problem \eqref{sys:Stokes.rot.res} naturally arises when 
studying the associated time-periodic problem 
\begin{equation}\label{sys:Stokes.rot.tp}
\begin{pdeq}
\pt\uvel+\tay\rotterm{\uvel} 
- \Delta \uvel
+ \grad \upres
 &= f
&& \tin \torus\times\Omega, \\
\Div\uvel&=0
&& \tin \torus\times\Omega, \\
\uvel&=0
&& \ton \torus\times\partial\Omega.
\end{pdeq}
\end{equation}
System \eqref{sys:Stokes.rot.tp} may be regarded as the linearization of 
the nonlinear problem
\begin{equation}\label{sys:NS.rot.tp}
\begin{pdeq}
\partial_t \uvel + \tay\rotterm{\uvel} 
+ \uvel\cdot\grad\uvel
&= f + \Delta \uvel - \grad \upres 
&& \tin \torus\times\Omega, \\
\Div\uvel&=0
&& \tin \torus\times\Omega, \\
\uvel&=\tay\eone\wedge x
&& \ton \torus\times\partial\Omega, \\
\lim_{\snorm{x}\to\infty} \uvel(t,x) &= 0
&& \tfor t\in \torus,
\end{pdeq}
\end{equation}
which
describes the time-periodic flow of a viscous incompressible fluid around 
a rotating rigid body $\calb\coloneqq\R^3\setminus\Omega$ in the three-dimensional space.
More precisely, here we assume that the fluid adheres to the boundary of $\calb$ 
and is at rest at infinity,
and that the body rotates about the $x_1$-axis
with (scalar) rotational velocity $\tay>0$.
Then the motion of the fluid flow,
described in a frame attached to the body, 
is governed by \eqref{sys:NS.rot.tp}.
The functions $\uvel\colon\torus\times\Omega\to\R^3$ and $\upres\colon\torus\times\Omega\to\R$ 
are velocity and pressure fields,
and $f\colon\torus\times\Omega\to\R^3$ is an external body force.
In \eqref{sys:Stokes.rot.tp} and \eqref{sys:NS.rot.tp} we choose $\torus\coloneqq\R/\per\Z$ for $\per>0$
as the time axis,
so that all occurring functions 
are intrinsically time periodic.
Observe that 
in the formulation of \eqref{sys:Stokes.rot.res} and \eqref{sys:Stokes.rot.tp}
we omitted the condition \eqrefsub{sys:NS.rot.tp}{4} at infinity,
which is later incorporated in the definition of the function spaces in a generalized sense.

Concerning the analysis of the nonlinear time-periodic problem \eqref{sys:NS.rot.tp},
the first result was given by Galdi and Silvestre
\cite{GaldiSilvestre_ExistenceTPSolutionsNSAroundMovingBody_2006},
who showed the existence of weak solution in the more general configuration 
where the rigid body performs a time-periodic motion.
However, in their functional framework the spatial asymptotic properties 
of the flow were not captured.
This problem was recently solved by Galdi \cite{Galdi2020_ExistenceUniquenessAsBehRegularTPViscFlowAroundMovingBodyRotCase},
who showed existence of regular solutions satisfying certain pointwise decay estimates.
A different approach to characterize the spatial behavior of solutions
is inspired by the fundamental work of Yamazaki \cite{Yamazaki2000},
who showed existence of time-periodic solutions to \eqref{sys:NS.rot.tp}
in the case $\tay=0$ in a framework of $\LR{3,\infty}$ spaces,
also known as weak-$\LR{3}$ spaces.
His analysis was based on well-known $\LR{p}$-$\LR{q}$ estimates for the Stokes semigroup.
For $\tay>0$ analogous $\LR{p}$-$\LR{q}$ estimates for the semigroup associated with the initial-value problem
corresponding to \eqref{sys:Stokes.rot.tp}
were shown 
by Hishida and Shibata \cite{HishidaShibata_LpLqEstimatesStokesOperatorRotatingObstacle_2009},
so that Yamazaki's method also leads to solutions to \eqref{sys:NS.rot.tp} in the $\LR{3,\infty}$ framework.
Later, 
Geissert, Hieber and Nguyen \cite{GeissertHieberNguyen_TP2016}
developed a semigroup-based approach in a general framework,
where this analysis was carried out as a special case
that lead to existence of mild time-periodic solutions to \eqref{sys:NS.rot.tp}.

With regard to the linearized time-periodic problem \eqref{sys:Stokes.rot.tp},
observe that for $\tay=0$, 
it reduces to the well-known Stokes problem.
In this case 
the unique existence of time-periodic solutions,
which satisfy suitable \textit{a priori} estimates,
was successfully derived
in \cite{GaldiKyed_TPflowViscLiquidpBody_2018}.
The aim of the present article is to establish
a similar result in the case $\tay>0$.
Observe that the additional rotation term $\tay\rotterm{\uvel}$ for $\tay>0$
cannot be treated as a lower-order perturbation of the Laplace operator
because the term $\eone\wedge x\cdot\grad$ 
is a differential operator with unbounded coefficient.
Therefore, for the derivation of \textit{a priori} estimates, 
this term has to be handled in a different way.
One suitable method was recently developed 
in \cite{KyedGaldi_asplqesoserfI,KyedGaldi_asplqesoserfII}
and is roughly described as follows:
The rotation term $\tay\rotterm{\uvel}$ in 
\eqref{sys:Stokes.rot.tp} and \eqref{sys:NS.rot.tp}
stems from the change of coordinates from an inertial frame
to a rotating frame. 
Undoing this transformation, 
one can simply absorb this term again.
However, in general
this leads to a problem on a time-dependent spatial domain. 
Therefore, the idea is to first employ this procedure in the setting of the whole space $\Omega=\R^3$,
where the domain is invariant, 
and to use cut-off techniques to return to the case of an exterior domain afterwards.
While in \cite{KyedGaldi_asplqesoserfI,KyedGaldi_asplqesoserfII} 
steady motions were investigated,
in the recent article \cite{EiterKyed_ViscousFlowAroundRigidBodyPerformingTPMotion_2021}
the described method was successfully applied to the time-periodic problem 
\begin{equation}\label{sys:Oseen.rot.tp}
\begin{pdeq}
\pt\uvel+\tay\rotterm{\uvel} 
+\rey\partial_1\uvel
- \Delta \uvel
+ \grad \upres
 &= f
&& \tin \torus\times\Omega, \\
\Div\uvel&=0
&& \tin \torus\times\Omega, \\
\uvel&=0
&& \ton \torus\times\partial\Omega
\end{pdeq}
\end{equation}
for $\rey\neq0$.
System \eqref{sys:Oseen.rot.tp} differs from \eqref{sys:Stokes.rot.tp}
by the term $\rey\partial_1\uvel$,
which arises when the body $\calb$ performs, besides a rotation, an additional translation
with velocity $\rey\eone$.
However, in \cite{EiterKyed_ViscousFlowAroundRigidBodyPerformingTPMotion_2021}
well-posedness of \eqref{sys:Oseen.rot.tp} with $\rey\neq0$
was merely shown under the restriction that the time period $\per$
and the angular velocity $\tay$ are related by $\perf=\tay$.
The results in the forthcoming paper \cite{Eiter_OseenResProblemTPFlowRotatingObstacle}
show that 
this restriction is not necessary
but can be weakened and replaced with $\perf/\tay\in\Q$.
One main observation of this article
is that in the present situation,
that is, in the case $\rey=0$,
where \eqref{sys:Oseen.rot.tp} reduces to \eqref{sys:Stokes.rot.tp},
such an assumption is not necessary at all,
and we provide a framework of well-posedness for any $\tay,\per>0$
without further restrictions.

To this end, 
the major part of the subsequent analysis is focused on the 
resolvent problem \eqref{sys:Stokes.rot.res}.
Observe that
if $\np{\uvel,\upres}$ is a $\per$-periodic solution to 
\eqref{sys:Stokes.rot.tp},
then the Fourier coefficient of order $k$ 
is a solution to \eqref{sys:Stokes.rot.res} with $s=\perf k$.
This explains why we restrict our analysis to
purely imaginary resolvent parameters $is$, $s\in\R$.
Moreover, since we want to choose arbitrary time periods $\per>0$,
we need well-posedness of the resolvent problem \eqref{sys:Stokes.rot.res} for all $s\in\R$.

At first glance, 
it may seem reasonable to analyze \eqref{sys:Stokes.rot.res}
as the resolvent problem $is\vvel+A_\tay\vvel=g$ of the closed operator 
$A_\tay\colon D(A_\tay)\subset\LRsigma{q}(\Omega)\to\LRsigma{q}(\Omega)$
given by
\begin{align}
D(A_\tay)&\coloneqq\setcl{\vvel\in\LRsigma{q}(\Omega)\cap\WSRN{1}{q}(\Omega)^3\cap\WSR{2}{q}(\Omega)^3}{\eone\wedge x\cdot\grad\vvel\in\LR{q}(\Omega)^3},
\label{eq:def.DAtay}
\\
A_\tay\vvel&\coloneqq \projhOm\bb{\tay\rotterm{\vvel} 
- \Delta \vvel},
\label{eq:def.Atay}
\end{align}
where $\LRsigma{q}(\Omega)$, $q\in(1,\infty)$ is the 
space of all solenoidal functions in $\LR{q}(\Omega)^3$,
and $\projhOm$ is the associated Helmholtz projection.
Farwig, Ne\v{c}asov\'{a} and Neustupa \cite{FarwigNecasovaNeustupa2007_EssSpectrumStokesOpRotatingBodyLq}
could show that the essential spectrum of $A_\tay$ is given by
\begin{equation}
\label{eq:spectrum.usual}
\sigma_{\text{ess}}\bp{A_\tay}=\setcl{\alpha+i\tay\ell}{\alpha\leq 0,\ \ell\in\Z}.
\end{equation}
In particular, 
we see that
$is$, $s\in\R$,
does not belong to the resolvent set of $A_\tay$ in general,
and this setting does not provide a framework for
well-posedness of \eqref{sys:Stokes.rot.res}
if $s\in\tay\Z$.
Since, as explained above, we need such a framework
in order to solve the time-periodic problem \eqref{sys:Stokes.rot.tp},
we introduce a different functional setting instead, 
namely a setting of homogeneous Sobolev spaces
that renders \eqref{sys:Stokes.rot.res} well posed 
for arbitrary $s\in\R$.
One peculiarity of the derived \textit{a priori} estimate is that
instead of the classical form
\begin{equation}\label{est:resest.usual}
\snorm{s}\norm{\vvel}_{q}+\norm{A_\tay\vvel}_{q}
\leq C\norm{g}_{q},
\end{equation} 
we deduce the non-classical resolvent estimate
\[
\norm{is\vvel+\tay\rotterm{\vvel}}_{q}
+\norm{\Delta \vvel}_{q}
\leq C\norm{g}_{q},
\] 
see \eqref{est:Stokes.rot.res} below.
In particular, we do not obtain separate estimates 
of the terms $is\vvel$ and $\tay\rotterm{\vvel}$
or even of $is\vvel$, $\tay\eone\wedge\vvel$ and $\tay\eone\wedge x\cdot\grad\vvel$.
This is not a surprise 
since a separate estimate of $is\vvel$ would lead to \eqref{est:resest.usual} for all $s\in\R$,
which would contradict \eqref{eq:spectrum.usual}.
Moreover, it is well known that separate estimates of $\tay\eone\wedge\vvel$ and $\tay\eone\wedge x\cdot\grad\vvel$
are not even feasible for the steady-state problem, 
that is, for \eqref{sys:Stokes.rot.res} with $s=0$;
see \cite[Theorem VIII.7.2]{GaldiBookNew} for example.

The analysis of the resolvent problem \eqref{sys:Stokes.rot.res}
for an exterior domain $\Omega\subset\R^3$
goes back to Hishida \cite{Hishida_StokesOperatorRotationEffectExteriorDomain_1999},
who derived suitable resolvent estimates in an $\LR{2}$ framework
that showed that the operator $A_\tay$ generates a contractive $\CR{0}$-semigroup if $q=2$.
For general $q\in(1,\infty)$ a similar statement in the $\LR{q}$ setting was later proved 
by Geissert, Heck and Hieber
\cite{GeissertHeckHieber_LpTheoryNSFlowExteriorMovingRotatingObstacle_2006}.
However, since
the resolvent estimate \eqref{est:resest.usual} is invalid on for all $is$ with $s\in\tay\Z$,
the operator $A_\tay$ does not generate an analytic semigroup.
Nevertheless,
one can derive additional smoothing properties of the semigroup
that allow to establish solutions to the nonlinear initial-value 
problem \cite{Hishida_ExistenceTheoremNavierStokesExteriorRotatingObstacle_1999,
GeissertHeckHieber_LpTheoryNSFlowExteriorMovingRotatingObstacle_2006}
and to carry out a stability analysis of steady-state solutions,
as was done by Hishida and Shibata \cite{HishidaShibata_DecayEstStokesFlowRotatingObstacle_2007,
HishidaShibata_LpLqEstimatesStokesOperatorRotatingObstacle_2009}.
Moreover, the investigation of the spectrum of the operator $A_\tay$ 
was further deepened by Farwig, Ne\v{c}asov\'{a} and Neustupa
\cite{FarwigNeustupa_SpectrumStokesTypeOpFlowAroundRotatingBody_2007, FarwigNecasovaNeustupa2007_EssSpectrumStokesOpRotatingBodyLq, FarwigNecasovaNeustupa2011_SpectralAnalysisStokesOpRotatingBody}.

As explained above, 
in our investigation of the resolvent problem \eqref{sys:Stokes.rot.res}
we follow a different approach
and investigate \eqref{sys:Stokes.rot.res} in a 
different functional framework.
Our analysis is based on the study of the auxiliary problem
\begin{equation}\label{sys:Stokes.tp.aux.intro}
is\uvel
+\pt\uvel
- \Delta \uvel
+ \grad \upres 
= f,
\quad
\Div\uvel
=0 
\qquad\tin\torus\times\R^3,
\end{equation}
which may be regarded as a mixture of the Stokes resolvent problem with the time-periodic Stokes problem.
In contrast to \eqref{sys:Stokes.rot.res} and \eqref{sys:Stokes.rot.tp},
we can directly derive a formula for the solution to \eqref{sys:Stokes.tp.aux.intro}
by means of a Fourier multiplier on $\torus\times\R^3$.
Using tools from harmonic analysis in this group setting,
we can further deduce suitable $\LR{q}$ estimates.
By means of the aforementioned transformation, 
we can then introduce the rotational terms
and relate the resolvent problem \eqref{sys:Stokes.rot.tp}
to problem \eqref{sys:Stokes.tp.aux.intro}.

This article is structured as follows: 
After introducing some notation in Section \ref{sec:Notation},
we state our main results on the well-posedness of
the resolvent problem
\eqref{sys:Stokes.rot.res} 
and
the time-periodic problem \eqref{sys:Stokes.rot.tp} 
in Section \ref{sec:MainResults}.
In Section \ref{sec:resprob.wholespace} 
we study the resolvent problem \eqref{sys:Stokes.rot.res}
in the case of the whole space $\Omega=\R^3$,
which is based on the examination of the auxiliary time-periodic problem
\eqref{sys:Stokes.tp.aux.intro}.
In the subsequent Section \ref{sec:resprob.extdom}
these findings are transferred to the case of an exterior domain.
Finally, in Section \ref{sec:tpprob}
we show the existence of a unique solution 
to the time-periodic problem \eqref{sys:Stokes.rot.tp}
in a framework of functions with absolutely convergent Fourier series.

\section{Notation}
\label{sec:Notation}
In order to state and prove our main results, 
we first introduce the basic notation.

The symbols $C$ and $C_j$ with $j\in\N$
always denote generic positive constants. 
We occasionally emphasize that $C$ depends on a specific set of quantities 
$\set{a,b\dots}$
by writing $C=C(a,b,\dots)$.

When we fix a time period $\per>0$,
the associated torus group is denoted by $\torus\coloneqq\R/\per\Z$.
Then every element of $\torus$ can be identified with 
a unique representative in $[0,\per)$, which we tacitly do from time to time.
Moreover, $\torus$ is always equipped with
the normalized Haar measure such that
\[
\forall f\in\CR{}(\torus):\quad
\int_\torus f(t)\,\dt
\coloneqq
\frac{1}{\per}
\int_0^\per f(t')\,\dt',
\]
where $\CR{}(\torus)$ is the class of continuous functions on $\torus$.
A point $(t,x)\in\torus\times\R^3$
is composed of a time variable $t\in\torus$ 
and a space variable $x=(x_1,x_2,x_3)\in\R^3$.
We denote the Euclidean norm of $x$ by $\snorm{x}$,
and
$x\cdot y$, $x\wedge y$ and $x\otimes y$
represent the scalar, vector and tensor products of $x,y\in\R^3$.
We further use the shorthand 
$x\wedge y\cdot z\coloneqq\np{x\wedge y}\cdot z$ for $x,y,z\in\R^3$.

Time and spatial derivatives are denoted by
$\pt$ and $\partial_j\coloneqq\partial_{x_j}$, $j=1,2,3$, respectively,
and the symbols for (spatial) gradient, divergence and Laplace operator are
$\grad$, $\Div$ and $\Delta$. 
The symbol $\grad^2\uvel$ denotes the collection of all 
second-order spatial derivatives of a sufficiently regular function $\uvel$.

In the whole article we either have $\Omega=\R^3$ or
we let $\Omega\subset\R^3$ be an exterior domain,
that is, $\Omega$ is a domain and its complement is a compact nonempty set in $\R^3$.
Moreover, $\ball_R\subset\R^3$ denotes the ball
of radius $R>0$ centered at $0$,
and $\Omega_R\coloneqq\Omega\cap\ball_R$.

For classical Lebesgue and Sobolev spaces we write $\LR{q}(\Omega)$
and $\WSR{k}{q}(\Omega)$, where $q\in[1,\infty]$ and $k\in\N$,
and $\norm{\cdot}_{q;\Omega}$ 
and $\norm{\cdot}_{k,q;\Omega}$ denote the associated norms. 
When the domain is clear from the context, 
we simply write $\norm{\cdot}_{q}$ 
and $\norm{\cdot}_{k,q}$ instead.
This convention is adapted for
the norm $\norm{\cdot}_{q,\torus\times\Omega}$
of the Lebesgue space $\LR{q}(\torus\times\Omega)$ in space and time.
We further let 
$\CRci(\Omega)$ be the class of all smooth functions with compact support in $\Omega$,
and $\WSRN{1}{q}(\Omega)$ denotes its closure in $\WSR{1}{q}(\Omega)$.
For the dual space of $\WSRN{1}{q}(\Omega)$ we write $\WSR{-1}{q'}(\Omega)$,
where $1/q+1/q'=1$,
which we equip with the norm $\norm{\cdot}_{-1,q';\Omega}$.
Moreover, $\LRloc{q}(\Omega)$ and $\WSRloc{k}{q}(\Omega)$
denote the classes of all functions that 
locally belong to $\LR{q}(\Omega)$ and $\WSR{k}{q}(\Omega)$,
respectively.

We usually do not distinguish between 
a space $X$ and its vector-valued version $X^n$, $n\in\N$,
when the dimension is clear from the context.
By $\norm{\cdot}_X$ we denote the norm of a general normed space $X$.
We write $\LR{q}(\torus;X)$ 
for the corresponding Bochner--Lebesgue space when $q\in[1,\infty)$,
and we define
$\WSR{1}{q}(\torus;X)\coloneqq\setcl{\uvel\in\LR{q}(\torus;X)}{\pt\uvel\in\LR{q}(\torus;X)}$.

In our subsequent analysis, the configuration for $\Omega=\R^3$ plays an important role.
In this case, the space-time domain is given by
$\grp\coloneqq\torus\times\R^3$,
which is a locally compact abelian group
with dual group isomorphic to $\dualgrp\coloneqq\Z\times\R^3$.
As natural generalizations of the classes of Schwartz functions 
and tempered distributions in the Euclidean setting,
one can define the Schwartz--Bruhat space $\SR(\grp)$
and its dual space $\TDR(\grp)$ on $\grp$,
which were first introduced by \textsc{Bruhat} \cite{Bruhat61},
see also \cite{EiterKyed_tplinNS_PiFbook} for more details and a precise definition
of these spaces.
In this framework, the Fourier transform $\FT_\grp$ and its inverse $\iFT_\grp$, 
defined by
\[
\begin{aligned}
\FT_\grp\colon\SR(\grp)\ra\SR(\dualgrp), 
&&\FT_\grp\nb{\uvel}(k,\xi)
&\coloneqq\int_\torus\int_{\R^3} \uvel(t,x)\e^{-ix\cdot\xi-ik t}\,\dx\dt,\\
\iFT_\grp\colon\SR(\dualgrp)\ra\SR(\grp), 
&&\iFT_\grp\nb{\wvel}(t,x)
&\coloneqq\sum_{k\in\Z}\,\int_{\R^3} \wvel(k,\xi)\e^{ix\cdot\xi+ik t}\,\dxi,
\end{aligned}
\]
are mutually inverse isomorphisms
provided that the Lebesgue measure $\dxi$ is normalized in a suitable way.
By duality, the Fourier transform can be extended to
an isomorphism $\TDR(\grp)\to\TDR(\dualgrp)$.
By analogy, the Fourier transforms 
on the groups $\torus$ and $\R^3$ are given by
\[
\begin{aligned}
\FT_{\torus}\colon\SR(\torus)&\ra\SR(\Z), 
&\qquad
\FT_{\torus}\nb{\uvel}(k)
&\coloneqq\int_\torus \uvel(t)\e^{-ik t}\,\dt,\\
\iFT_{\torus}\colon\SR(\Z)&\ra\SR(\torus), 
&\qquad
\iFT_{\torus}\nb{\wvel}(t)
&\coloneqq\sum_{k\in\Z}\wvel(k)\e^{ik t},
\\
\FT_{\R^3}\colon\SR(\R^3)&\ra\SR(\R^3), 
&\qquad
\FT_{\R^3}\nb{\uvel}(\xi)
&\coloneqq\int_{\R^3} \uvel(x)\e^{-ix\cdot\xi}\,\dx,\\
\iFT_{\R^3}\colon\SR(\R^3)&\ra\SR(\R^3), 
&\qquad
\iFT_{\R^3}\nb{\wvel}(x)
&\coloneqq\int_{\R^3} \wvel(\xi)\e^{ix\cdot\xi}\,\dxi.
\end{aligned}
\]

Our investigation of the time-periodic problem \eqref{sys:Stokes.rot.tp}
will mainly be performed in a framework of
spaces of absolutely convergent Fourier series.
For a normed space $X$,
these are defined by
\begin{align}\label{eq:def.Aspace}
\begin{aligned}
\AR(\torus;X)
&\coloneqq
\setcL{f\colon\torus\to X}{f(t)=\sum_{k\in\Z}f_k \e^{ikt}, \ f_k\in X, \ 
\sum_{k\in\Z}\norm{f_k}_{X}<\infty},
\\
\norm{f}_{\AR(\torus;X)}
&\coloneqq\sum_{k\in\Z}\norm{f_k}_{X}.
\end{aligned}
\end{align}
If $X$ is a Banach space,
then $\AR(\torus;X)$ coincides with the Banach space
$\iFT_\torus\bb{\lR{1}(\Z;X)}$, 
and $\AR(\torus;X)\embeds\CR{}(\torus;X)$.
Observe that many inequalities in spaces $X$ 
have natural extensions to the corresponding spaces $\AR(\torus;X)$,
for example, H\"older's inequality or interpolation inequalities;
see 
\cite[Prop.~3.1 and 3.2]{EiterKyed_ViscousFlowAroundRigidBodyPerformingTPMotion_2021}.
We also use the shorthand
$\uvel\in\AR(\torus;\WSRloc{k}{q}(\Omega))$
when $\uvel\in\AR(\torus;\WSR{k}{q}(K))$
for all compact sets $K\subset\Omega$.

The existence of solutions to 
the time-periodic problem 
\eqref{sys:Stokes.rot.tp} will be established in the following functional framework.
We fix $\tay>0$ and $q\in(1,3/2)$.
Then the space for the velocity field is given by 
\[
\begin{aligned}
\XT(\torus\times\Omega)\coloneqq
\setcl{\uvel\in\AR(\torus;\WSRloc{2}{q}(\Omega)^3)}{
{}&\grad^2\uvel, \,
\pt\uvel+\tay\rotterm{\uvel} \in\AR(\torus;\LR{q}(\Omega)),\\
&
\uvel\in\AR(\torus;\LR{3q/(3-2q)}(\Omega)), \,
\grad\uvel\in\AR(\torus;\LR{3q/(3-q)}(\Omega))},
\end{aligned}
\]
and the function class for the pressure term is given by
\[
\YT(\torus\times\Omega)
\coloneqq
\setcl{\upres\in\AR(\torus;\WSRloc{1}{q}(\Omega))}{
\grad\upres\in\AR(\torus;\LR{q}(\Omega)),\
\upres\in\AR(\torus;\LR{3q/(3-q)}(\Omega))}.
\]
Similarly, we introduce the following function classes for solutions to the resolvent problem
\eqref{sys:Stokes.rot.res}.
For $\tay>0$, $q\in(1,3/2)$ and $s\in\R$,
we define the class of velocity fields by
\[
\begin{aligned}
\Xs(\Omega)\coloneqq
\setcl{\vvel\in\WSRloc{2}{q}(\Omega)^3}{
{}&\grad^2\vvel, \,
is\vvel+\tay\rotterm{\vvel} \in\LR{q}(\Omega),\\
&\qquad\quad
\vvel\in\LR{3q/(3-2q)}(\Omega), \ 
\grad\vvel\in\LR{3q/(3-q)}(\Omega)},
\end{aligned}
\]
and the corresponding pressure belongs to
\[
\Ys(\Omega)
\coloneqq
\setcl{\vpres\in\WSRloc{1}{q}(\Omega)}{
\grad\vpres\in\LR{q}(\Omega),\
\vpres\in\LR{3q/(3-q)}(\Omega)}.
\]
Observe that the function class $\Xs(\Omega)$ for the velocity field
also depends on the resolvent parameter $s\in\R$.
Moreover, if
$\uvel$ belongs to $\XreyT(\torus\times\Omega)$, 
then its $k$-th Fourier coefficient 
$\uvel_k\coloneqq\FT_\torus\nb{\uvel}(k)$
belongs to $\Xs(\Omega)$ with $s=\perf k$.

\section{Main Results}
\label{sec:MainResults}

The main results of this article concern the question of well-posedness
of the time-periodic linear problem \eqref{sys:Stokes.rot.tp}
and the associated resolvent problem \eqref{sys:Stokes.rot.res}.
At first, we address the resolvent problem.

\begin{thm}\label{thm:Stokes.rot.res}
Let $\Omega=\R^3$ or $\Omega\subset\R^3$ 
be an exterior domain with $\CR{3}$-boundary.
Let $s\in\R$ and $0<\tay\leq\tay_0$,
and let $q\in(1,3/2)$ and $g\in\LR{q}(\Omega)^3$.
Then there exists a unique solution 
$\np{\vvel,\vpres}\in\Xs(\Omega)\times\Ys(\Omega)$
to \eqref{sys:Stokes.rot.res}
that obeys the estimates
\begin{equation}\label{est:Stokes.rot.res}
\begin{aligned}
&\norm{\dist(s,\tay\Z) \,\vvel}_q
+\norm{is\vvel+\tay\rotterm{\vvel}}_{q}
+\norm{\grad^2 \vvel}_{q}
+\norm{\grad \vpres}_{q}
\\
&\qquad\qquad\qquad\qquad\qquad
+\norm{\grad \vvel}_{3q/(3-q)}
+\norm{\vvel}_{3q/(3-2q)}
+\norm{\vpres}_{3q/(3-q)}
\leq C\norm{g}_{q}
\end{aligned}
\end{equation}
for a constant 
$C=C(\Omega,q,\tay_0)>0$.
In particular, $C$ can be chosen independently of $s\in\R$ and $\tay\in(0,\tay_0]$.
\end{thm}

Note that 
if $s\not\in\tay\Z$,
then estimate \eqref{est:Stokes.rot.res} implies 
$\vvel\in\LR{q}(\Omega)$,
which yields $\vvel\in D(A_\tay)$,
where $D(A_\tay)$ is defined in \eqref{eq:def.DAtay}.
But a similar inclusion cannot be obtained if $s\in\tay\Z$.
This observation is in complete accordance with \eqref{eq:spectrum.usual}.

Working within a framework of absolutely convergent Fourier series,
we can then employ Theorem \ref{thm:Stokes.rot.res}
on the level of the Fourier coefficients
to derive well-posedness of the time-periodic problem \eqref{sys:Stokes.rot.tp}.
As will become clear from the proof, 
to conclude existence of $\per$-periodic solutions,
it is important that the constant $C$ in \eqref{est:Stokes.rot.res}
can be chosen uniformly for $s\in\perf\Z$.

\begin{thm}\label{thm:Stokes.rot.tp}
Let $\Omega=\R^3$ or $\Omega\subset\R^3$ 
be an exterior domain with $\CR{3}$-boundary.
Let $\per>0$ and $0<\tay\leq\tay_0$,
and let $q\in(1,3/2)$ and $f\in\LR{q}(\torus\times\Omega)^3$.
Then there exists a unique $\per$-periodic solution 
$\np{\uvel,\upres}\in\XT(\torus\times\Omega)\times\YT(\torus\times\Omega)$
to \eqref{sys:Stokes.rot.tp}
that obeys the estimates
\begin{equation}\label{est:Stokes.rot.tp}
\begin{aligned}
&\norm{\pt\uvel+\tay\rotterm{\uvel}}_{\AR(\torus;\LR{q}(\Omega))}
+\norm{\grad^2 \uvel}_{\AR(\torus;\LR{q}(\Omega))}
+\norm{\grad \upres}_{\AR(\torus;\LR{q}(\Omega))}
\\
&\qquad\qquad\quad
+\norm{\grad \uvel}_{\AR(\torus;\LR{3q/(3-q)}(\Omega))}
+\norm{\uvel}_{\AR(\torus;\LR{3q/(3-2q)}(\Omega))}
+\norm{\upres}_{\AR(\torus;\LR{3q/(3-q)}(\Omega))}
\\
&\qquad\qquad\qquad\qquad\qquad\qquad\qquad\qquad\qquad\qquad\qquad\qquad\qquad
\leq C\norm{f}_{\AR(\torus;\LR{q}(\Omega))}
\end{aligned}
\end{equation}
for a constant $C=C(\Omega,q,\tay_0)>0$.
In particular, $C$ can be chosen independently of $\tay\in(0,\tay_0]$ and $\per>0$.
\end{thm}

\begin{rem}\label{rem:est.dist}
In contrast to the other terms in estimate \eqref{est:Stokes.rot.res},
the term $\norm{\dist(s,\tay\Z) \,\vvel}_q$ 
does not directly correspond to any of the terms in \eqref{est:Stokes.rot.tp}.
However, going through the proof of Theorem \ref{thm:Stokes.rot.tp},
one may derive an additional estimate.
Decompose the 
set of Fourier indices into 
$A_1\coloneqq\setc{k\in\Z}{\perf k\in\tay\Z}$
and $A_2\coloneqq\setc{k\in\Z}{\perf k\not\in\tay\Z}$
and split the velocity field $\uvel$ accordingly
as 
\[
\uvel=\uvel^{(1)}+\uvel^{(2)}, 
\qquad
\uvel^{(1)}\coloneqq\sum_{k\in A_1}\uvel_k \e^{i\perf kt},
\qquad
\uvel^{(2)}\coloneqq\sum_{k\in A_2}\uvel_k \e^{i\perf kt}.
\]
Then the estimate
\[
\norml{ d_{\tay,\per} \,\uvel^{(2)}}_{\AR(\torus;\LR{q}(\Omega))}
\leq C\norm{f}_{\AR(\torus;\LR{q}(\Omega))}
\]
follows, where
\[
d_{\tay,\per} 
\coloneqq\inf\setcL{\dist\bp{\perf k,\tay\Z}}{k\in\Z,\,\perf k\not\in\tay\Z}
=\inf\setcL{\snorm{a}}{0\neq a\in\perf\Z+\tay\Z}.
\]
Of course, this estimate only provides new information when $d_{\tay,\per}>0$.
A classical argument shows that this is the case if and only if 
$\perf/\tay\in\Q$.
\end{rem}

\section{The resolvent problem in the whole space}
\label{sec:resprob.wholespace}

We begin with the study
of the resolvent problem \eqref{sys:Stokes.rot.res}
in the case $\Omega=\R^3$,
where it simplifies to
\begin{equation}\label{sys:Stokes.rot.res.R3}
\begin{pdeq}
is\vvel
+\tay\rotterm{\vvel}
- \Delta \vvel
+ \grad \vpres 
&= g
&&\tin\R^3,
\\
\Div\vvel
&=0 
&& \tin\R^3.
\end{pdeq}
\end{equation}
In this section we show the following result on well-posedness of \eqref{sys:Stokes.rot.res.R3}.

\begin{thm}\label{thm:Stokes.rot.res.R3}
Let $\tay>0$ and $s\in\R$. 
For each $g\in\LR{q}(\R^3)^3$ there exists 
a solution 
$\np{\vvel,\vpres}\in\WSRloc{2}{q}(\R^3)^3
\times\WSRloc{1}{q}(\R^3)$ 
to \eqref{sys:Stokes.rot.res.R3}
that satisfies
\begin{equation}\label{est:Stokes.rot.res.R3}
\norm{\dist(s,\tay\Z)\,\vvel}_{q}
+\norm{is\vvel+\tay\rotterm{\vvel}}_{q}
+\norm{\grad^2 \vvel}_{q}
+\norm{\grad \vpres}_{q}
\leq C\norm{g}_{q}
\end{equation}
as well as
\begin{align}
\label{est:Stokes.rot.res.R3.grad}
\norm{\grad\vvel}_{3q/(3-q)}
+\norm{\vpres}_{3q/(3-q)}
&\leq C\norm{g}_{q}
\qquad\tif q<3,
\\
\label{est:Stokes.rot.res.R3.fct}
\norm{\vvel}_{3q/(3-2q)}
&\leq C\norm{g}_{q}
\qquad\tif q<3/2,
\end{align}
for a constant $C=C(q)>0$.
Moreover, if $\np{\wvel,\wpres}\in\LRloc{1}(\torus\times\R^3)^{3+1}$
is another distributional solution
to \eqref{sys:Stokes.rot.res.R3}, then the following holds:
\begin{enumerate}[label=\roman*.]
\item
If $\grad^2\wvel,\, is\wvel+\tay\rotterm{\wvel}\in\LR{q}(\R^3)$,
then 
\[
\begin{aligned}
is\wvel+\tay\rotterm{\wvel}&=is\vvel+\tay\rotterm{\vvel},
\\
\grad^2\wvel=\grad^2\vvel,
\qquad
\grad\wpres&=\grad\vpres.
\end{aligned}
\]
\item
If $q<3/2$ or $s\not\in\tay\Z$, and if  
$\wvel\in\LR{r}(\R^3)^3$ for some $r\in(1,\infty)$,
then $\vvel=\wvel$ and $\vpres=\wpres+c$ for a constant $c\in\R$.
\end{enumerate}
\end{thm}

In order to prove Theorem \ref{thm:Stokes.rot.res.R3},
we first consider the 
auxiliary problem
\begin{equation}\label{sys:Stokes.tp.mod.R3}
\begin{pdeq}
is\uvel
+\pt\uvel
- \Delta \uvel
+ \grad \upres 
&= f
&&\tin\torus\times\R^3,
\\
\Div\uvel
&=0 
&& \tin\torus\times\R^3,
\end{pdeq}
\end{equation}
which can be regarded as a mixture of 
the Stokes resolvent problem and the time-periodic Stokes problem.
In contrast to the original time-periodic problem \eqref{sys:Stokes.rot.tp},
the differential operator associated with \eqref{sys:Stokes.tp.mod.R3}
has constant coefficients,
which enables us to
express its solution via Fourier multipliers.
Since \eqref{sys:Stokes.tp.mod.R3}
is a problem in the locally compact abelian group $\grp\coloneqq\torus\times\R^3$,
we work with multiplier arguments in this group framework.
This is more involved compared to the usual Euclidean setting,
and for a detailed introduction to this theory, 
with a focus on the analysis of the Navier--Stokes equations,
we refer to the book chapter \cite{EiterKyed_tplinNS_PiFbook}
as well as to the monographs \cite{Kyed_habil,Eiter_Diss}.
The main tool for the derivation of $\LR{q}$ multiplier estimates 
is the so-called transference principle for multipliers,
which goes back to de Leuuw \cite{Leeuw1965}
and was generalized by 
Edwards and Gaudry \cite[Theorem B.2.1]{EdwardsGaudryBook}.
In our investigation we employ the following special case.

\begin{thm}
\label{thm:TransferencePrinciple}
Let $\grp\coloneqq\torus\times\R^3$ and $\grph\coloneqq\R\times\R^3$.
For each $q\in(1,\infty)$ there exists a constant $C_q>0$
with the following property:
If a continuous, bounded function $M\colon\grph\to\C$ is 
an $\LR{q}\np{\grph}$ multiplier,
that is,
\[
\forall h\in\SR(\grph) : \quad
\norml{\iFT_{\grph}\bb{M\,\FT_{\grph}\nb{h}}}_{\LR{q}(\grph)}
\leq C_M \norm{h}_{\LR{q}(\grph)}
\]
for some $C_M>0$,
then the restriction $m\coloneqq\restriction{M}{\Z\times\R^3}$
is an $\LR{q}\np{\grp}$ multiplier with
\[
\forall g\in\SR(\grp) : \quad
\norml{\iFT_{\grp}\bb{m\,\FT_{\grp}\nb{g}}}_{\LR{q}(\grp)}
\leq C_q C_M \norm{g}_{\LR{q}(\grp)}.
\]
\end{thm}

This result enables us to reduce Fourier multipliers in $\grp=\torus\times\R^3$ to Fourier multipliers
in the Euclidean space $\grph=\R\times\R^3$, 
where more classical tools for the identification of $\LR{q}$ multipliers
are available.
This strategy is used several times in the proof of the following theorem
that establishes existence of solutions to \eqref{sys:Stokes.tp.mod.R3} 
in an $\LR{q}$ framework.

\begin{thm}\label{thm:Stokes.tp.mod.R3}
Let $\tay>0$ and $s\in\R$, and set $\per_\tay\coloneqq\frac{2\pi}{\tay}$ and
$\torus\coloneqq\R/\per_\tay\Z$.
For each $f\in\LR{q}(\torus\times\R^3)^3$ there exists 
a solution $\np{\uvel,\upres}$ to \eqref{sys:Stokes.tp.mod.R3} with
\[
\uvel\in\WSR{1}{q}(\torus;\LRloc{q}(\R^3)^3)\cap\LR{q}(\torus;\WSRloc{2}{q}(\R^3)^3),
\qquad
\upres\in\LR{q}(\torus;\WSRloc{1}{q}(\R^3))
\]
that satisfies
\begin{equation}\label{est:Stokes.tp.mod.R3}
\norm{\dist(s,\tay\Z)\,\uvel}_{q}
+\norm{is\uvel+\pt\uvel}_{q}
+\norm{\grad^2 \uvel}_{q}
+\norm{\grad \upres}_{q}
\leq C\norm{f}_{q}
\end{equation}
as well as
\begin{align}
\label{est:Stokes.tp.mod.R3.grad}
\norm{\grad\uvel}_{\LR{q}(\torus;\LR{3q/(3-q)}(\R^3))}
+\norm{\upres}_{\LR{q}(\torus;\LR{3q/(3-q)}(\R^3))}
&\leq C\norm{f}_{q}
\qquad\tif q<3,
\\
\label{est:Stokes.tp.mod.R3.fct}
\norm{\uvel}_{\LR{q}(\torus;\LR{3q/(3-2q)}(\R^3))}
&\leq C\norm{f}_{q}
\qquad\tif q<3/2,
\end{align}
for a constant $C=C(q)>0$.
Moreover, if $\np{\wvel,\wpres}\in\LRloc{1}(\torus\times\R^3)^{3+1}$
is another distributional solution
to \eqref{sys:Stokes.tp.mod.R3}, then the following holds:
\begin{enumerate}[label=\roman*.]
\item
\label{it:Stokes.tp.mod.R3.uniqueness.i}
If $\grad^2\wvel,\ is\wvel+\pt\wvel\in\LR{q}(\torus\times\R^3)$,
then 
\[
is\wvel+\pt\wvel=is\uvel+\pt\uvel,
\qquad
\grad^2\wvel=\grad^2\uvel,
\qquad
\grad\wpres=\grad\upres.
\]
\item
\label{it:Stokes.tp.mod.R3.uniqueness.ii}
If $q<3/2$ or $s\not\in\tay\Z$, and if 
$\wvel\in\LR{1}(\torus;\LR{r}(\R^3)^3)$ for some $r\in(1,\infty)$,
then $\uvel=\wvel$ and $\upres=\wpres+d$ for a (space-independent) function 
$d\colon\torus\to\R$.
\end{enumerate}
\end{thm}

\begin{proof}
At first, we show that it suffices to consider $s\in\R$ with $\snorm{s}\leq\tay/2$.
Indeed, for $s\in\R$ there exists $\ell\in\Z$ such that
$\snorm{s-\tay\ell}\leq\tay/2$.
We set $\widetilde{s}=s-\tay\ell$ and $\widetilde{f}(t,x)=f(t,x)\e^{i\tay\ell t}$,
and assume that $\np{\wvel,\wpres}$
is a solution to \eqref{sys:Stokes.tp.mod.R3} 
satisfying \eqref{est:Stokes.tp.mod.R3}--\eqref{est:Stokes.tp.mod.R3.fct}
with $s$ and $f$ replaced with $\widetilde{s}$ and $\widetilde{f}$,
respectively.
We now define the $\per_\tay$-periodic functions 
$\uvel(t,x)\coloneqq\wvel(t,x)\e^{-i\tay\ell t}$
and $\upres(t,x)\coloneqq\wpres(t,x)\e^{-i\tay\ell t}$.
Then 
$\np{\uvel,\upres}$
satisfies the original problem \eqref{sys:Stokes.tp.mod.R3} 
and the corresponding estimates \eqref{est:Stokes.tp.mod.R3}--\eqref{est:Stokes.tp.mod.R3.fct}.
Therefore, it is sufficient to treat the case $\snorm{s}\leq\tay/2$ in the following.

In the case $s=0$, the system \eqref{sys:Stokes.tp.mod.R3}
reduces to the classical time-periodic Stokes system,
for which existence of a solution $\np{\uvel,\upres}$
was shown in \cite{Kyed_mrtpns}.
More precisely, in \cite{Kyed_mrtpns} the right-hand side was 
decomposed as $f=f_0+f_\bot$
with 
\begin{equation}\label{eq:dec.f}
f_0(x)=\int_{\torus} f(t,x)\,\dt, \qquad
f_\bot(t,x)=f(t,x)-f_0(x),
\end{equation}
and existence of a solution 
$\np{\uvel,\upres}=\np{\uvel_0+\uvel_\bot,\upres_0+\upres_\bot}$,
decomposed in the same fashion as $f$, 
was shown,
which satisfies
\[
\begin{aligned}
\norm{\grad^2\uvel_0}_{\LR{q}(\R^3)}
+\norm{\grad\upres_0}_{\LR{q}(\R^3)}
&\leq C_1 \norm{f_0}_{\LR{q}(\R^3)},
\\
\norm{\pt\uvel_\bot}_{\LR{q}(\torus\times\R^3)}
+\norm{\grad^2\uvel_\bot}_{\LR{q}(\torus\times\R^3)}
+\norm{\grad\upres_\bot}_{\LR{q}(\torus\times\R^3)}
&\leq C_2 \norm{f_\bot}_{\LR{q}(\torus\times\R^3)}
\end{aligned}
\]
for constants $C_1=C_1(q)$ and $C_2=C_2(q,\per_\tay)$.
In particular, combining these inequalities,
we end up with \eqref{est:Stokes.tp.mod.R3}
with the constant $C=C_1+C_2$.
Moreover, by a classical scaling argument
we see that the constant $C$ in \eqref{est:Stokes.tp.mod.R3}
is independent of $\per_\tay$.
Finally, \eqref{est:Stokes.tp.mod.R3.grad} and \eqref{est:Stokes.tp.mod.R3.fct}
follow from Sobolev's inequality in space.

Now let us consider the case $0\neq\snorm{s}\leq\tay/2$.
At first, let $f\in\SR(\torus\times\R^3)^3$.
Computing the divergence of \eqrefsub{sys:Stokes.tp.mod.R3}{1},
we obtain $\Delta\upres=\Div f$.
By means of the Fourier transform $\FT_\grp$ 
on the locally compact abelian group $\grp\coloneqq\torus\times\R^3$,
we conclude $-\snorm{\xi}^2\FT_\grp\nb{\upres}=i\xi\cdot\FT_\grp\nb{f}$,
so that
\begin{equation}\label{eq:Stokes.tp.mod.R3.gradp}
\upres
=\iFT_{\grp}\Bb{\frac{-i\xi}{\snorm{\xi}^2}\FT_\grp\nb{f}},
\qquad
\grad\upres
=\iFT_{\grp}\Bb{\frac{\xi\otimes\xi}{\snorm{\xi}^2}\FT_\grp\nb{f}}.
\end{equation}
In particular, $\upres$ is well defined as a distribution in $\TDR(\grp)$, 
and by continuity of the Riesz transforms $\LR{q}(\R^3)\to\LR{q}(\R^3)$, 
which can be extended to continuous operators 
$\LR{q}(\grp)\to\LR{q}(\grp)$,
we conclude
\begin{equation}\label{est:Stokes.tp.mod.R3.gradp}
\norm{\grad\upres}_{q}\leq C\norm{f}_{q}.
\end{equation}
Next we apply the Fourier transform to \eqrefsub{sys:Stokes.tp.mod.R3}{1}.
In view of \eqref{eq:Stokes.tp.mod.R3.gradp},
this leads to the representation formula
\begin{equation}\label{eq:Stokes.tp.mod.R3.u}
\uvel=\iFT_\grp\bb{m\,\FT_\grp\nb{f-\grad\upres}}
=\iFT_\grp\Bb{m\Bp{\idmatrix-\frac{\xi\otimes\xi}{\snorm{\xi}^2}}\FT_\grp\nb{f}}
\end{equation}
where $\idmatrix\in\R^{3\times3}$ is the identity matrix and
\[
m\colon\Z\times\R^3\to\R,\qquad
m(k,\xi)\coloneqq\frac{1}{is+i\tay k +\snorm{\xi}^2}.
\]
Since $0\neq\snorm{s}\leq\tay/2$,
the denominator 
\[
D_s(k,\xi)\coloneqq is+i\tay k +\snorm{\xi}^2
\]
in the definition of $m$ has no zeros $(k,\xi)\in\Z\times\R^3$,
so that $m$ is a well-defined bounded function.
Hence $\uvel\in\TDR(\grp)$ is well defined by means of a Fourier multiplier in $\grp$.
To deduce estimate \eqref{est:Stokes.tp.mod.R3},
it remains to derive $\LR{q}$ estimates for  
$\dist(s,\tay\Z)\,\uvel$, $is\uvel+\partial_t\uvel$ and $\grad^2\uvel$,
that is,
estimates of $is\uvel$, $\partial_t\uvel$ and $\grad^2\uvel$.
In virtue of the representation formula \eqref{eq:Stokes.tp.mod.R3.u}
we have
\begin{equation}\label{eq:Stokes.tp.mod.R3.uder}
\begin{aligned}
is\uvel
&=\iFT_\grp\Bb{m_0\Bp{\idmatrix-\frac{\xi\otimes\xi}{\snorm{\xi}^2}}\FT_\grp\nb{f}},
\\
\pt\uvel
&=\iFT_\grp\Bb{m_1\Bp{\idmatrix-\frac{\xi\otimes\xi}{\snorm{\xi}^2}}\FT_\grp\nb{f}},
\\
\partial_j\partial_\ell\uvel
&=\iFT_\grp\Bb{m_{j\ell}\Bp{\idmatrix-\frac{\xi\otimes\xi}{\snorm{\xi}^2}}\FT_\grp\nb{f}},
\end{aligned}
\end{equation}
with $m_0,\,m_1,\,m_{j\ell}\colon\Z\times\R^3\to\R$ defined by
\[
m_0(k,\xi)\coloneqq\frac{is}{D_s(k,\xi)},
\qquad
m_1(k,\xi)\coloneqq\frac{i\tay k}{D_s(k,\xi)},
\qquad
m_{j\ell}(k,\xi)\coloneqq\frac{-\xi_j\xi_\ell}{D_s(k,\xi)}
\]
for $j,\ell=1,2,3$.
We set $\mathfrak m
\coloneqq\setcl{m_0,m_1,m_{j\ell}}{j,\ell\in\set{1,2,3}}$.
Then
the $\LR{q}$ estimate \eqref{est:Stokes.tp.mod.R3} follows 
if all 
$\widetilde{m}\in\mathfrak m$
can be identified
as $\LR{q}(\grp)$ multipliers.
For this purpose,
we employ the transference principle from Theorem \ref{thm:TransferencePrinciple}.
Let $\cutoff\in\CRi(\R)$ with $0\leq\cutoff\leq1$ and
such that $\cutoff(x)=0$ for $\snorm{x}\leq 1/2$ and $\cutoff(x)=1$ for $\snorm{x}\geq1$.
We define the functions $M_0,\,M_1,\,M_{j\ell}\colon\R\times\R^3\to\C$,
\[
M_0(\eta,\xi)\coloneqq\frac{is\,\cutoff\bp{1+\frac{\tay\eta}{s}}}{D_s(\eta,\xi)}, 
\quad
M_1(\eta,\xi)\coloneqq\frac{i\tay\eta\,\cutoff\bp{1+\frac{\tay\eta}{s}}}{D_s(\eta,\xi)}, 
\quad
M_{j\ell}(\eta,\xi)\coloneqq\frac{\xi_j\xi_\ell\,\cutoff\bp{1+\frac{\tay\eta}{s}}}{D_s(\eta,\xi)},
\]
and we set $\mathfrak M\coloneqq\setcl{M_0,M_1,M_{j\ell}}{j,\ell\in\set{1,2,3}}$.
Observe that 
\[
\cutoff\bp{1+\frac{\tay\eta}{s}}=0 \quad \tif \snorm{s+\tay\eta}\leq \snorm{s}/2,
\qquad
\cutoff\bp{1+\frac{\tay\eta}{s}}=1 \quad \tif \snorm{s+\tay\eta}\geq \snorm{s},
\]
so that the numerator of each term vanishes in a neighborhood of the 
only zero $(\eta,\xi)=(-s/\tay,0)$ of the common denominator $D_s(\eta,\xi)$.
We thus conclude that every multiplier 
$\widetilde M\in\mathfrak M$
is a well-defined continuous function.
Moreover, we have $m_0=\restriction{M_0}{\Z\times\R^3}$,
$m_1=\restriction{M_1}{\Z\times\R^3}$ and
$m_{j\ell}=\restriction{M_{j\ell}}{\Z\times\R^3}$.
Hence, by the transference principle from Theorem \ref{thm:TransferencePrinciple},
all elements of $\mathfrak m$ 
are $\LR{q}(\grp)$ multipliers
if all elements of $\mathfrak M$ are $\LR{q}(\R\times\R^3)$ multipliers.
By employing the technical inequalities
\[
\begin{aligned}
\snorm{s+\tay\eta}\geq \snorm{s}/2 
&\ \implies \
\snormL{\frac{\tay\eta}{D_s(\eta,\xi)}}
\leq \snormL{\frac{\tay\eta}{s+\tay\eta}}
\leq 1 + \snormL{\frac{s}{s+\tay\eta}}
\leq 3,
\\
\snorm{s+\tay\eta}\leq \snorm{s} 
&\ \implies \ 
\snormL{\frac{\tay\eta}{s}}
\leq 1 + \snormL{\frac{s+\tay\eta}{s}}
\leq 2,
\end{aligned}
\]
a lengthy but elementary calculation shows
\[
\sup
\setcl{\snorml{
\eta^{\alpha}\xi^{\beta}\partial_\eta^{\alpha}\partial_\xi^\beta
\widetilde{M}(\eta,\xi)
}}
{\alpha\in\set{0,1},\,
\beta\in\set{0,1}^3,\,
\np{\eta,\xi}\in\R\times\R^3
}
\leq C
\]
for all $\widetilde M\in\mathfrak{M}$ and
for an absolute constant $C>0$ that is independent of $s$.
By the Marcinkiewicz multiplier theorem 
(see \cite[Corollary 5.2.5]{Grafakos1} for example),
we conclude that $\widetilde{M}$ is an $\LR{q}(\R\times\R^3)$ multiplier
such that
\[
\forall h\in\SR(\R\times\R^3):\quad
\norm{\iFT_{\R\times\R^3}\bb{\widetilde{M}\FT_{\R\times\R^3}\nb{h}}}_{q}
\leq C \norm{h}_{q},
\]
where $C=C(q)$ is independent of $s$.
Since all $\widetilde M\in\mathfrak{M}$ are continuous, 
the transference principle (Theorem \ref{thm:TransferencePrinciple})
now implies that all $\widetilde{m}\in\mathfrak{m}$
are $\LR{q}(\grp)$ multipliers with
\[
\forall g\in\SR(\grp):\quad
\norm{\iFT_{\grp}\bb{\widetilde{m}\,\FT_{\grp}\nb{g}}}_{q},
\leq C \norm{g}_{q}
\]
where $C=C(q)$.
By combining these estimates with the continuity of the Riesz transforms,
the representation formulas in \eqref{eq:Stokes.tp.mod.R3.uder}
yield \eqref{est:Stokes.tp.mod.R3}.
Estimates \eqref{est:Stokes.tp.mod.R3.grad} and \eqref{est:Stokes.tp.mod.R3.fct}
now follow from Sobolev's inequality.
In summary, 
for $f\in\SR(\grp)$ we have now constructed a solution to \eqref{sys:Stokes.tp.mod.R3}
with the desired properties.
A classical approximation argument
based on the estimates \eqref{est:Stokes.tp.mod.R3}--\eqref{est:Stokes.tp.mod.R3.fct} 
finally yields the existence of a solution for any $f\in\LR{q}(\grp)$. 

It remains to prove the uniqueness assertion for arbitrary $s\in\R$.
We consider the difference 
$\np{\tuvel,\tupres}=\np{\uvel-\wvel,\upres-\wpres}\in\LRloc{1}(\grp)^{3+1}$,
which is a solution  
to \eqref{sys:Stokes.tp.mod.R3} with $f=0$.
As above, computing the divergence of both sides of \eqrefsub{sys:Stokes.tp.mod.R3}{1},
we conclude $\Delta\tupres=0$, which implies
$\supp\FT_\grp\nb{\tupres}\subset\Z\times\set{0}$.
Since an application of $\FT_\grp$ to \eqrefsub{sys:Stokes.tp.mod.R3}{1}
leads to
$D_s(k,\xi)\FT_\grp\nb{\tuvel}=-i\xi\FT_\grp\nb{\tupres}$
with $D_s(k,\xi)= is+i\tay k +\snorm{\xi}^2$,
we deduce
\[
\supp\bb{D_s(k,\xi)\FT_\grp\nb{\tuvel}}
\subset\Z\times\set{0}.
\]
Since $D_s(k,\xi)$ can only vanish for $\xi=0$, 
we deduce
$\supp\FT_\grp\nb{\tuvel}
\subset\Z\times\set{0}$.
Hence, $\tuvel\in\LRloc{1}(\grp)$
implies $\supp\FT_{\R^3}\nb{\tuvel}(t,\cdot)\subset\set{0}$
for a.a.~$t\in\torus$,
so that $\tuvel(t,\cdot)$ is a polynomial 
for a.a.~$t\in\torus$. 
In the same way we show that $\upres(t,\cdot)$ is a polynomial for a.a.~$t\in\torus$.
This has the following consequences 
in the two distinguished cases.
In case \ref{it:Stokes.tp.mod.R3.uniqueness.i}~
we have $\grad^2\tuvel,\,\grad\tupres\in\LR{q}(\grp)$.
Since both are polynomials in space a.e.~in $\torus$,
this is only possible if $\grad^2\tuvel=0$ and $\grad\tupres=0$.
In virtue of \eqrefsub{sys:Stokes.tp.mod.R3}{1}, 
this also implies $is\tuvel+\pt\tuvel=0$,
which shows the assertion in this case.
In case \ref{it:Stokes.tp.mod.R3.uniqueness.ii}~
we have $\tuvel\in\LR{1}(\torus;\LR{r_0}(\R^3)^3+\LR{r}(\R^3)^3)$
with $r_0=3q/(3-2q)$ if $q<3/2$, and $r_0=q$ if $s\not\in\tay\Z$.
Since $\tuvel(t,\cdot)$ is a polynomial 
for a.a.~$t\in\torus$,
this is only possible if $\tuvel=0$,
and returning to \eqrefsub{sys:Stokes.tp.mod.R3}{1}, we also conclude $\grad\upres=0$.
In total, this completes the proof.
\end{proof}

Now let us consider the modified time-periodic Stokes problem with rotating effect
\begin{equation}\label{sys:Stokes.tp.rot.mod.R3}
\begin{pdeq}
is\uvel
+\pt\uvel
+\tay\rotterm{\uvel}
- \Delta \uvel
+ \grad \upres 
&= f
&&\tin\torus\times\R^3,
\\
\Div\uvel
&=0 
&& \tin\torus\times\R^3,
\end{pdeq}
\end{equation}
which differs from \eqref{sys:Stokes.tp.mod.R3} 
by the rotational term $\tay\rotterm{\uvel}$.
In the particular case that the angular velocity of the rotation $\tay$ 
coincides with the angular frequency $\perf$ associated to the time period $\per$,
this rotational term can be absorbed in the time derivative by a suitable transformation.
In this way, we reduce \eqref{sys:Stokes.tp.rot.mod.R3} to \eqref{sys:Stokes.tp.mod.R3},
and we transfer existence and uniqueness as well as \textit{a priori} estimates 
from Theorem \ref{thm:Stokes.tp.mod.R3}.
Observe that the restriction $\tay=\perf$ is crucial for this procedure.

\begin{thm}\label{thm:Stokes.tp.rot.mod.R3}
Let $\tay>0$ and $s\in\R$, and set $\per_\tay\coloneqq\frac{2\pi}{\tay}$ and
$\torus\coloneqq\R/\per_\tay\Z$.
For each $f\in\LR{q}(\torus\times\R^3)^3$ there exists 
a solution $\np{\uvel,\upres}$ with
\[
\uvel\in\WSR{1}{q}(\torus;\LRloc{q}(\R^3)^3)\cap\LR{q}(\torus;\WSRloc{2}{q}(\R^3)^3),
\qquad
\upres\in\LR{q}(\torus;\WSRloc{1}{q}(\R^3))
\]
to \eqref{sys:Stokes.tp.rot.mod.R3}
that satisfies
\begin{equation}\label{est:Stokes.tp.rot.mod.R3}
\norm{\dist(s,\tay\Z)\,\uvel}_{q}
+\norm{is\uvel+\pt\uvel+\tay\rotterm{\uvel}}_{q}
+\norm{\grad^2 \uvel}_{q}
+\norm{\grad \upres}_{q}
\leq C\norm{f}_{q}
\end{equation}
as well as
\eqref{est:Stokes.tp.mod.R3.grad} and \eqref{est:Stokes.tp.mod.R3.fct}
for a constant $C=C(q)>0$.

Moreover, if $\np{\wvel,\wpres}\in\LRloc{1}(\torus\times\R^3)^{3+1}$
is another distributional solution
to \eqref{sys:Stokes.tp.rot.mod.R3}, then the following holds:
\begin{enumerate}[label=\roman*.]
\item
If $\grad^2\wvel,\ is\wvel+\pt\wvel+\tay\rotterm{\wvel}\in\LR{q}(\torus\times\R^3)$,
then 
\[
\begin{aligned}
is\wvel+\pt\wvel+\tay\rotterm{\wvel}&=is\uvel+\pt\uvel+\tay\rotterm{\uvel},
\\
\grad^2\wvel=\grad^2\uvel,
\qquad&\qquad
\grad\wpres=\grad\upres.
\end{aligned}
\]
\item
If $q<3/2$ or $s\not\in\tay\Z$, and if 
$\wvel\in\LR{1}(\torus;\LR{r}(\R^3)^3)$ for some $r\in(1,\infty)$,
then $\uvel=\wvel$ and $\upres=\wpres+d$ for a (space-independent) function 
$d\colon\torus\to\R$.
\end{enumerate}
\end{thm}

\begin{proof}
The proof is based on the idea 
to absorb the rotational term $\tay\rotterm{\uvel}$ into the time derivative
by the coordinate transform
arising from the rotation matrix 
\begin{equation}\label{eq:Q.def}
\rotmatrix_\tay(t)\coloneqq\begin{pmatrix}
1 & 0 & 0\\
0 & \cos(\tay t) & -\sin(\tay t)\\
0 & \sin(\tay t) & \cos(\tay t)
\end{pmatrix}.
\end{equation}
Let $f\in\LR{q}(\torus\times\R^3)^3$ and define the vector field $\widetilde f$ by
\[
\widetilde{f}(t,x)
\coloneqq\rotmatrix_\tay(t) f(t,\rotmatrix_\tay(t)^\transpose x).
\]
Then $\widetilde{f}\in\LR{q}(\torus\times\R^3)^3$ 
since $\torus=\R/\per_\tay\Z$ with $\per_\tay=\frac{2\pi}{\tay}$.
By Theorem \ref{thm:Stokes.tp.mod.R3}
there exists a solution 
$\np{\tuvel,\tupres}$
to \eqref{sys:Stokes.tp.mod.R3} (with $f$ replaced by $\widetilde f$),
which satisfies the estimates \eqref{est:Stokes.tp.mod.R3}--\eqref{est:Stokes.tp.mod.R3.fct}.
We now define the $\per_\tay$-time-periodic functions
\[
\uvel(t,x)\coloneqq\rotmatrix_\tay(t)^\transpose\tuvel(t,\rotmatrix_\tay(t) x),
\qquad
\upres(t,x)\coloneqq\tupres(t,\rotmatrix_\tay(t) x).
\]
Since 
$\dot{\rotmatrix}_\tay(t)x
= \tay\eone\wedge \nb{\rotmatrix_\tay(t)x}
= \rotmatrix_\tay(t)\nb{\tay\eone\wedge x}$ 
for any $x\in\R^3$, a direct computation shows
\[
\begin{aligned}
&\pt\tuvel(t,x)
\\
&\quad
=\rotmatrix_\tay(t)
\bb{\pt\uvel(t,\rotmatrix_\tay(t)^\transpose x)
+\tay\eone\wedge\uvel(t,\rotmatrix_\tay(t)^\transpose x)
-\tay\eone\wedge\nb{\rotmatrix_\tay(t)^\transpose x}\cdot\grad\uvel(t,\rotmatrix_\tay(t)^\transpose x)}.
\end{aligned}
\]
Moreover, we have
\[
\begin{aligned}
\Delta \tuvel(t,x) 
&= \rotmatrix_\tay(t)\Delta\uvel(t,\rotmatrix_\tay(t)^\transpose x),
&
\partial_1\tuvel(t,x)
&=\rotmatrix_\tay(t)\partial_1\uvel(t,\rotmatrix_\tay(t)^\transpose x),
\\
\grad \tupres(t,x) 
&= \rotmatrix_\tay(t)\grad\upres(t,\rotmatrix_\tay(t)^\transpose x),
&\qquad
\Div\tuvel(t,x)
&=\Div\uvel(t,\rotmatrix_\tay(t)^\transpose x).
\end{aligned}
\]
Consequently, $\np{\uvel,\upres}$ is a solution to \eqref{sys:Stokes.tp.rot.mod.R3}
and satisfies the estimates \eqref{est:Stokes.tp.rot.mod.R3}, 
\eqref{est:Stokes.tp.mod.R3.grad}, \eqref{est:Stokes.tp.mod.R3.fct}.

For the uniqueness statement, 
we set
\[
\begin{aligned}
\tuvel(t,x)&\coloneqq\rotmatrix_\tay(t)\uvel(t,\rotmatrix_\tay(t)^\transpose x),
&\qquad
\tupres(t,x)&\coloneqq\upres(t,\rotmatrix_\tay(t)^\transpose x) 
\\
\twvel(t,x)&\coloneqq\rotmatrix_\tay(t)\wvel(t,\rotmatrix_\tay(t)^\transpose x),
&\qquad
\twpres(t,x)&\coloneqq\wpres(t,\rotmatrix_\tay(t)^\transpose x),
\\
\widetilde f(t,x)&\coloneqq\rotmatrix_\tay(t) f(t,\rotmatrix_\tay(t)^\transpose x).
\end{aligned}
\]
Mimicking the above calculations, we see that $\np{\tuvel,\tupres}$ 
and $\np{\twvel,\twpres}$ are solutions
to \eqref{sys:Stokes.tp.mod.R3} with the same right-hand side $\widetilde f$.
The uniqueness statement 
now follows from the corresponding statement in Theorem \ref{thm:Stokes.tp.mod.R3}.
\end{proof}

Observe that, by simply considering 
$s=0$ in \eqref{sys:Stokes.tp.rot.mod.R3},
we would obtain the original time-periodic problem 
\eqref{sys:Stokes.rot.tp},
and Theorem \ref{thm:Stokes.tp.rot.mod.R3}
yields existence of a unique solution to this problem.
However, since
we required $\tay=\perf$ in Theorem \ref{thm:Stokes.tp.rot.mod.R3},
we only obtain well-posedness in this special case,
and $\tay$ and $\per$ cannot be chosen independently.

In contrast, if we consider time-independent solutions $\np{\uvel,\upres}(t,x)=\np{\vvel,\vpres}(x)$
to \eqref{sys:Stokes.tp.rot.mod.R3} for $s\in\R$,
we obtain the resolvent problem
\eqref{sys:Stokes.rot.res.R3},
where the $\per$-dependence does not appear anymore.
From Theorem \ref{thm:Stokes.tp.rot.mod.R3}
we can thus extract Theorem \ref{thm:Stokes.rot.res.R3} as the final result of this section.

\begin{proof}[Proof of Theorem \ref{thm:Stokes.rot.res.R3}]
For the proof we set $\torus\coloneqq\R/\per\Z$ with $\per=\frac{2\pi}{\tay}$.
At first, let $g\in\LR{q}(\R^3)^3$
and define $f(t,x)\coloneqq g(x)$. 
Then $f\in\LR{q}(\torus\times\R^3)^3$,
and by Theorem \ref{thm:Stokes.tp.rot.mod.R3} there exists a solution $\np{\uvel,\upres}$
to \eqref{sys:Stokes.tp.rot.mod.R3}.
Then
\[
\vvel(x)\coloneqq\int_\torus\uvel(t,x)\,\dt,
\qquad
\vpres(x)\coloneqq\int_\torus\upres(t,x)\,\dt,
\]
defines a solution $\np{\vvel,\vpres}$
to \eqref{sys:Stokes.rot.res.R3},
and estimates 
\eqref{est:Stokes.rot.res.R3}--\eqref{est:Stokes.rot.res.R3.fct}
follow directly.
With regard to uniqueness, observe that every solution to \eqref{sys:Stokes.rot.res.R3}
is a (time-independent) 
solution to \eqref{sys:Stokes.tp.rot.mod.R3},
so that the uniqueness statement follows immediately from 
Theorem \ref{thm:Stokes.tp.rot.mod.R3}.
\end{proof}

\section{The resolvent problem in an exterior domain}
\label{sec:resprob.extdom}

After having established well-posedness 
of the resolvent problem \eqref{sys:Stokes.rot.res.R3}
in $\R^3$,
we next consider the corresponding problem
in an exterior domain $\Omega\subset\R^3$,
given in \eqref{sys:Stokes.rot.res}.
The aim of this section is a proof of Theorem \ref{thm:Stokes.rot.res}.
At first,
we address the question of uniqueness
by considering \eqref{sys:Stokes.rot.res} for $g=0$.

\begin{lem}\label{lem:Oseen.rot.res.uniqueness}
Let $\Omega\subset\R^3$ be an exterior domain of class $\CR{1,1}$.
Let $\tay>0$, $s\in\R$, 
and let $\np{\vvel,\vpres}$ be a distributional solution to \eqref{sys:Stokes.rot.res}
with $g=0$ and
$\grad^2\vvel,\,\partial_1\vvel,\,\grad\vpres \in\LR{q}(\Omega)$
for some $r\in(1,\infty)$ and $\vvel\in\LR{r}(\Omega)$ for some $s\in(1,\infty)$.
Then 
$\vvel=0$ and $\vpres$ is constant.
\end{lem}

\begin{proof}
The proof follows exactly as in 
\cite[Lemma 5.6]{EiterKyed_ViscousFlowAroundRigidBodyPerformingTPMotion_2021}, 
where the statement was shown 
for the case $s\in\tay\Z$. 
Therefore, we only give a brief sketch here.
The idea is to employ a cut-off argument
that leads to a Stokes problem on a bounded domain
and to the resolvent problem \eqref{sys:Stokes.rot.res}
in the whole space,
both with error terms on the right-hand side.
Using classical elliptic regularity 
of the Stokes problem
and regularity properties 
for \eqref{sys:Stokes.rot.res}
established in Theorem \ref{thm:Stokes.rot.res.R3},
one can then show that
\[
\begin{aligned}
\forall r\in(1,\infty): &\quad 
is\vvel+\rottermsimple{\vvel},\,\grad^2\vvel,\,\grad\vpres
\in\LR{r}(\Omega),
\\
\forall r\in\big(\frac{3}{2},\infty]: &\quad
\grad\vvel\in\LR{r}(\Omega), 
\\
\forall r\in(3,\infty] : &\quad
\vvel\in\LR{r}(\Omega).
\end{aligned}
\]
Next we multiply 
\eqrefsub{sys:Stokes.rot.res}{1}
by $\vvel^\ast$, the complex conjugate of $\vvel$.
The above regularities enable us
to integrate the resulting identity over $\Omega_R$
and to pass to the limit $R\to\infty$.
Arguing as in \cite[Lemma 5.6]{EiterKyed_ViscousFlowAroundRigidBodyPerformingTPMotion_2021},
one obtains
\[
0 
= is\int_{\Omega}\snorm{\vvel}^2\,\dx
+\int_{\Omega}
\snorm{\grad\vvel}^2 \,\dx. 
\]
This yields $\grad\vvel=0$
and, in view of the imposed boundary conditions,
$\vvel=0$.
From \eqrefsub{sys:Stokes.rot.res}{1}
we finally conclude $\grad\vpres=0$, which completes the proof.
\end{proof}

In the next step we derive suitable \textit{a priori} estimates by a cut-off procedure. 
We begin with the following intermediate result.
For simplicity, we only consider the case $q<\frac{3}{2}$.

\begin{lem}\label{lem:Oseen.rot.res.estimates.pert}
Let $\Omega\subset\R^3$ be an exterior domain of class $\CR{1,1}$,
and $\rey\geq 0$, $\tay>0$ and $s\in\R$. 
Let $q\in(1,\infty)$ and $g\in\LR{q}(\Omega)^3$.
Consider a solution $\np{\vvel,\vpres}$ 
to \eqref{sys:Stokes.rot.res} that satisfies
\[
is\vvel+\tay\rotterm\vvel, \
\grad^2\vvel,\
\rey\partial_1\vvel,\
\grad\vpres
\in\LR{q}(\Omega)^3
\]
and $\vvel\in\LR{r_1}(\Omega)^3$, $\vpres\in\LR{\overline r_1}(\Omega)$ for some 
$r_1,\overline r_1\in(1,\infty)$.
Fix $R>0$ such that $\partial\Omega\subset\ball_{R}$.
If $\rey=0$ and $q\in(1,3/2)$,
then $\np{\vvel,\vpres}$ satisfies the estimate
\begin{equation}
\begin{aligned}
&\norm{\dist(s,\tay\Z) \,\vvel}_q
+\norm{is\vvel+\tay\rotterm\vvel}_{q}
+\norm{\grad^2 \vvel}_{q}
+\norm{\grad \vpres}_{q}
\\
&\qquad\qquad
+\norm{\grad\vvel}_{3q/(3-q)}
+\norm{\vvel}_{3q/(3-2q)}
+\norm{\vpres}_{3q/(3-q)}
\\
&\qquad\qquad\qquad\qquad
\leq C \bp{
\norm{g}_{q}+\np{1+\tay}\norm{\vvel}_{1,q;\Omega_{R}}+\norm{\vpres}_{q;\Omega_{R}}
+\snorm{s}\norm{\vvel}_{-1,q;\Omega_{R}}
}
\end{aligned}
\label{est:Stokes.rot.res.extdom.pert}
\end{equation}
for a constant $C=C(q,\Omega,R)>0$.
\end{lem}

\begin{proof}
Estimate \eqref{est:Stokes.rot.res.extdom.pert}
can be shown by a classical cut-off procedure.
We skip the details here
and refer to 
\cite[Lemma 5.7]{EiterKyed_ViscousFlowAroundRigidBodyPerformingTPMotion_2021},
where the related resolvent problem
\begin{equation}\label{sys:Oseen.rot.res}
\begin{pdeq}
is\vvel+\tay\rotterm{\vvel} 
+\rey\partial_1\vvel
- \Delta \vvel
+ \grad \vpres
 &= g
&& \tin \Omega, \\
\Div\vvel&=0
&& \tin \Omega, \\
\vvel&=0
&& \ton \partial\Omega
\end{pdeq}
\end{equation}
in the case $\rey>0$ and $s\in\tay\Z$ was considered.
In the present situation one may proceed in the very same way
by invoking estimates \eqref{est:Stokes.rot.res.R3}--\eqref{est:Stokes.rot.res.R3.fct}
from Theorem \ref{thm:Stokes.rot.res.R3}
as well as the uniqueness result from Lemma \ref{lem:Oseen.rot.res.uniqueness}.
\end{proof}

Based on a compactness argument,
we now show how to omit the error terms on the right-hand side of
\eqref{est:Stokes.rot.res.extdom.pert}
and to infer the estimate \eqref{est:Stokes.rot.res}.

\begin{lem}\label{lem:Stokes.rot.res.estimates}
In the situation of Lemma \ref{lem:Oseen.rot.res.estimates.pert} let $\rey=0$, 
$q\in(1,3/2)$
and $\tay\in(0,\tay_0]$ for some $\tay_0>0$.
Then $\np{\vvel,\vpres}$ satisfies estimate
\eqref{est:Stokes.rot.res}
for a constant $C=C(q,\Omega,\tay_0)>0$.
\end{lem}

\begin{proof}
We employ a contradiction argument
and assume that there exists no constant $C>0$ with 
the claimed properties such that \eqref{est:Stokes.rot.res} holds.
Then there exist sequences $\np{s_j}\subset\R$, $\np{\tay_j}\subset(0,\tay_0]$,
$\np{\vvel_j}\subset\WSRloc{2}{q}(\Omega)^3$, 
$\np{\vpres_j}\subset\WSRloc{1}{q}(\Omega)$,
$\np{g_j}\subset\LR{q}(\Omega)^3$
with 
\[
\dist(s_j,\tay_j\Z)\vvel_j, \,
is_j\vvel_j+\tay_j\rotterm{\vvel_j}, \,
\grad^2\vvel_j,\,
\grad\vpres_j
\in\LR{q}(\Omega)^3, 
\]
such that
\begin{equation}\label{eq:Stokes.rot.res.extdom.seqnorm}
\begin{aligned}
&\norm{\dist(s_j,\tay_j\Z)\vvel_j}_{q}
+\norm{is_j\vvel_j{+}\tay_j\rotterm{\vvel_j}}_{q}
+\norm{\grad^2 \vvel_j}_{q}
+\norm{\grad \vpres_j}_{q}
\\
&\qquad\qquad\qquad\qquad\qquad\qquad
+\norm{\grad\vvel_j}_{3q/(3-q)}
+\norm{\vvel_j}_{3q/(3-2q)}
+\norm{\vpres_j}_{3q/(3-q)}
=1
\end{aligned}
\end{equation}
and 
\[
\lim_{j\to\infty}\norm{g_j}_{q}
=0.
\]
Moreover, there exist sequences $(r_j)\subset(1,\infty)$, $(\overline r_j)\subset(1,\infty)$
such that $\vvel_j\in\LR{r_j}(\Omega)^3$, $\vpres_j\in\LR{\overline r_j}(\Omega)$ for all $j\in\N$.
Observe that the left-hand side of \eqref{eq:Stokes.rot.res.extdom.seqnorm}
is finite by Lemma \ref{lem:Oseen.rot.res.estimates.pert}
and can thus be normalized as in \eqref{eq:Stokes.rot.res.extdom.seqnorm}.
By the choice of a suitable subsequence,
we may assume that $\tay_j\to\tay\in[0,\tay_0]$,
$s_j\to s\in[-\infty,\infty]$ 
and $\dist(s_j,\tay_j\Z)\to \delta\in[-\tay_0/2,\tay_0/2]$ 
as $j\to\infty$.
For the moment fix $R>0$ with $\partial\Omega\subset\ball_R$.
In virtue of \eqref{eq:Stokes.rot.res.extdom.seqnorm}
and the estimate
\begin{equation}\label{est:isjvj}
\begin{aligned}
&\norm{is_j\vvel_j}_{q;\Omega_R}\\
&\quad\leq \norm{is_j\vvel_j+\tay_j\rotterm{\vvel_j}}_{q;\Omega_R}
+\norm{\tay_j\rotterm{\vvel_j}}_{q;\Omega_R}
\\
&\quad\leq \norm{is_j\vvel_j+\tay_j\rotterm{\vvel_j}}_{q} 
+ \tay_0\bp{\norm{\vvel_j}_{q;\Omega_R}
+ R\norm{\grad\vvel_j}_{q;\Omega_R}},
\end{aligned}
\end{equation}
the sequences 
$(\restriction{is_j\vvel_j}{\Omega_R})$, 
$(\restriction{\vvel_j}{\Omega_R})$ and
$(\restriction{\vpres_j}{\Omega_R})$
are bounded in 
$\LR{q}(\Omega_R)$, $\WSR{2}{q}(\Omega_R)$ and $\WSR{1}{q}(\Omega_R)$, respectively.
Upon selecting suitable subsequences,
we thus obtain the existence of $\wvel\in\LRloc{q}(\Omega)^3$, 
$\vvel\in\WSRloc{2}{q}(\Omega)^3$ 
and $\vpres\in\WSRloc{1}{q}(\Omega)$
such that
\[
is_j\vvel_j\wto\wvel 
\quad\tin\LR{q}(\Omega_R),
\qquad
\vvel_j\wto\vvel
\quad\tin\WSR{2}{q}(\Omega_R),
\qquad
\vpres_j\wto\vpres
\quad\tin\WSR{1}{q}(\Omega_R).
\]
By a Cantor diagonalization argument,
we obtain a subsequence such that the limit functions 
$\wvel$, $\vvel$, $\vpres$ are independent of the choice of $R$.
Moreover, 
the uniform bounds from \eqref{eq:Stokes.rot.res.extdom.seqnorm}
imply weak convergence of a subsequence in the corresponding spaces,
which implies
\[
\begin{aligned}
&\norm{\delta\vvel}_q
+\norm{\wvel+\tay\rotterm{\vvel}}_{q}
+\norm{\grad^2 \vvel}_{q}
+\norm{\grad \vpres}_{q}\\
&\qquad\qquad\qquad
+\norm{\grad\vvel}_{3q/(3-q)}
+\norm{\vvel}_{3q/(3-2q)}
+\norm{\vpres}_{3q/(3-q)}
\leq 1
\end{aligned}
\]
Firstly, we can now perform the limit $j\to\infty$ 
in \eqref{sys:Stokes.rot.res} 
($\vvel$, $\vpres$, $g$ is replaced with $\vvel_j$, $\vpres_j$, $g_j$)
and deduce
\begin{equation}\label{sys:Stokes.rot.res.limit}
\begin{pdeq}
\wvel +\tay\rotterm\vvel
- \Delta \vvel
+ \grad \vpres 
&= 0
&&\tin\Omega,
\\
\Div\vvel
&=0 
&& \tin\Omega,
\\
\vvel
&=0 
&& \tin\partial\Omega.
\end{pdeq}
\end{equation}
Secondly,
the compactness of the embeddings
$\WSR{2}{q}(\Omega_R)\embeds\WSR{1}{q}(\Omega_R)
\embeds\LR{q}(\Omega_R)\embeds\WSRN{-1}{q}(\Omega_R)$
implies the strong convergence
\[
is_j\vvel_j\to\wvel 
\quad\tin\WSR{-1}{q}(\Omega_R),
\qquad
\vvel_j\to\vvel
\quad\tin\WSR{1}{q}(\Omega_R),
\qquad
\vpres_j\to\vpres
\quad\tin\LR{q}(\Omega_R).
\]
By Lemma \ref{lem:Oseen.rot.res.estimates.pert}
we have \eqref{est:Stokes.rot.res.extdom.pert}
with $\vvel$, $\vpres$, $g$ replaced with $\vvel_j$, $\vpres_j$, $g_j$.
Employing \eqref{eq:Stokes.rot.res.extdom.seqnorm}
and passing to the limit $j\to\infty$ in this inequality 
leads to
\begin{equation}\label{est:Stokes.rot.res.extdom.limit.below}
1
\leq C \bp{
\np{1+\tay}\norm{\vvel}_{1,q;\Omega_{R}}+\norm{\vpres}_{q;\Omega_{R}}
+\norm{\wvel}_{-1,q;\Omega_{R}}
}.
\end{equation}
We now distinguish the following cases:
\begin{enumerate}[label=\roman*.]
\item
If $\snorm{s}<\infty$ and $\tay=0$, then $\wvel=is\vvel$
and \eqref{sys:Stokes.rot.res.limit}
simplifies to the classical Stokes resolvent problem with resolvent parameter $is$.
If $s\neq0$, this yields $is\vvel=\Delta\vvel-\grad\vpres\in\LR{q}(\Omega)$, 
so that $\vvel\in\WSR{2}{q}(\Omega)$.
Uniqueness in this functional framework is well known,
so that $\vvel=\grad\vpres=0$; see \cite{FarwigSohr1994} for example.
If $s=0$, then \eqref{sys:Stokes.rot.res.limit} 
is the steady-state Stokes problem and $\vvel\in\LR{3q/(3-2q)}(\Omega)$
implies $\vvel=\grad\vpres=0$
as follows from \cite[Theorem V.4.6]{GaldiBookNew} 
for example.
\item
If $\snorm{s}<\infty$ and $\tay>0$,
then $\wvel=is\vvel$ and \eqref{sys:Stokes.rot.res.limit}
coincides with \eqref{sys:Stokes.rot.res} with $g=0$.
Employing Lemma \ref{lem:Oseen.rot.res.uniqueness} and $\vvel\in\LR{3q/(3-q)}(\Omega)$,
we conclude $\vvel=\grad\vpres=0$.
\item
If $\snorm{s}=\infty$,
we note that for every $R>0$ such that $\partial\Omega\subset\ball_R$,
estimate \eqref{est:isjvj}
implies $\norm{\vvel_j}_{q;\Omega_R}\leq C_R /\snorm{s_j}$
for some $R$-dependent constant $C$.
Passing to the limit $j\to\infty$ 
and employing that $R$ was arbitrary,
we deduce $\vvel=0$ in $\Omega$, and \eqref{sys:Stokes.rot.res.limit}
reduces to $\wvel+\grad\vpres=0$,
which, in particular, yields $\wvel\in\LR{q}(\Omega)$.
Since we also have $\Div\wvel=0$ and $\restriction{\wvel}{\partial\Omega}=0$,
this equality 
corresponds to the Helmholtz decomposition in $\LR{q}(\Omega)$ of the zero function.
By uniqueness of this decomposition,
we conclude $\wvel=\grad\vpres=0$.
\end{enumerate} 
Finally, in all three cases we obtain
$\wvel=\vvel=\grad\vpres=0$,
which also yields $\vpres=0$ due to $\vpres\in\LR{3q/(3-q)}(\Omega)$.
In total, 
this is a contradiction to inequality \eqref{est:Stokes.rot.res.extdom.limit.below}
and finishes the proof.
\end{proof}

After the derivation of suitable \textit{a priori} estimates
in Lemma \ref{lem:Stokes.rot.res.estimates},
we next show the existence of a solution to the resolvent problem \eqref{sys:Stokes.rot.res}
for a sufficiently smooth right-hand side $g$.

\begin{lem}
\label{lem:Oseen.rot.res.exist.smooth}
Let $\Omega\subset\R^3$ be an exterior domain of class $\CR{3}$.
Let $\rey\geq0$, $\tay>0$, $s\in\R$ and $g\in\CRci(\Omega)^3$.
Then there exists a solution $\np{\vvel,\vpres}$ to \eqref{sys:Stokes.rot.res}
with
\[
\forall q\in(1,3/2): \
\np{\vvel,\vpres}\in\Xs(\Omega)\times\Ys(\Omega).
\]
\end{lem}

\begin{proof}
Existence for the related resolvent problem \eqref{sys:Oseen.rot.res}
in the case  $s\in\tay\Z$ and $\rey>0$ was shown in
\cite[Lemma 5.11]{EiterKyed_ViscousFlowAroundRigidBodyPerformingTPMotion_2021}
in full detail based on energy estimates and an ``invading domains'' technique
together with $\LR{q}$ estimates similar to \eqref{est:Stokes.rot.res}.
The proof for \eqref{sys:Stokes.rot.res}
for general $s\in\R$, which means \eqref{sys:Oseen.rot.res} for $\rey=0$,
follows along the same lines,
which is why we only give a rough sketch here.

First of all, we choose $R>0$ such that $\partial\Omega\subset\ball_R$.
For $m\in\N$ with $m>R$ we first consider the resolvent problem \eqref{sys:Stokes.rot.res} on 
the bounded domain $\Omega_m=\Omega\cap\ball_m$,
that is, 
\[
\begin{pdeq}
is\vvel_m+\tay\rotterm{\vvel_m} 
- \Delta \vvel_m
+ \grad \vpres_m
 &= g
&& \tin\Omega_m, \\
\Div\vvel_m&=0
&& \tin\Omega_m, \\
\vvel_m&=0
&& \ton \partial\Omega_m.
\end{pdeq}
\]
By formally testing with the complex conjugates of $\vvel_m$ and $\projhm\Delta\vvel_m$,
where $\projhm$ denotes the Helmholtz projection in $\LR{2}(\Omega_m)$,
one can then derive the \textit{a priori} estimates
\[
\begin{aligned}
\norm{\vvel_m}_{6;\Omega_m}+\norm{\grad\vvel_m}_{2;\Omega_m}
&\leq C \norm{g}_{6/5},
\\
\norm{\projhm\Delta\vvel_m}_{2;\Omega_m}
&\leq C \bp{\norm{g}_{6/5}+\norm{g}_{2}},
\end{aligned}
\]
where the constant $C>0$ is independent of $m$; 
see the proof of \cite[Lemma 5.11]{EiterKyed_ViscousFlowAroundRigidBodyPerformingTPMotion_2021}
for further details.
In order to derive a uniform estimate on the full second-order norm,
we employ the inequality
\[
\norm{\grad^2\wvel}_{2;\Omega_m}
\leq C \np{\norm{\projhm\Delta\wvel}_{2;\Omega_m}
+\norm{\grad\wvel}_{2;\Omega_m}}
\]
for all $\wvel\in\WSRN{1}{2}(\Omega_m)^3\cap\WSR{2}{2}(\Omega_m)^3$ with $\Div\wvel=0$.
Since we assumed $\partial\Omega\in\CR{3}$,
the constant $C$ can be chosen independent of $m$;
see \cite[Lemma 1]{Heywood1980}.
Based on these formal \textit{a priori} estimates
and a basis of eigenfunctions of the Stokes operator on the bounded domain $\Omega_m$,
we can then apply a Galerkin method
to conclude the existence of a solution $\np{\vvel_m,\vpres_m}$,
which satisfies the \textit{a priori} estimate
\[
\norm{\vvel_m}_{6;\Omega_m}+\norm{\grad\vvel_m}_{1,2;\Omega_m}
\leq C \np{\norm{g}_{6/5}+\norm{g}_{2}},
\]
where $C$ is independent of $m$.
After multiplication with suitable cut-off functions,
one can then pass to the limit $m\to\infty$,
which leads to 
a solution $\np{\vvel,\vpres}$ to the original resolvent problem 
\eqref{sys:Stokes.rot.res}.
Finally, another cut-off argument that uses
the uniqueness properties from Lemma \ref{lem:Oseen.rot.res.uniqueness}
reveals that $\np{\vvel,\vpres}\in\Xs(\Omega)\times\Ys(\Omega)$ for all $q\in(1,3/2)$.
\end{proof}

After having shown existence of a solution for smooth data $g$,
we can now combine the previous lemmas
to conclude the proof of Theorem \ref{thm:Stokes.rot.res} by an approximation argument.

\begin{proof}[Proof of Theorem \ref{thm:Stokes.rot.res}]
In the case $\Omega=\R^3$ the statement follows from Theorem \ref{thm:Stokes.rot.res.R3}
above.
In the case of an exterior domain $\Omega\subset\R^3$,
the uniqueness statement is a consequence of Lemma \ref{lem:Oseen.rot.res.uniqueness},
and estimate \eqref{est:Stokes.rot.res}
was shown in Lemma \ref{lem:Stokes.rot.res.estimates}.
It thus remains to show existence of a solution for general $g\in\LR{q}(\Omega)^3$.
To this end, consider a sequence $(g_j)\subset\CRci(\Omega)^3$
converging to $g$ in $\LR{q}(\Omega)^3$.
By Lemma \ref{lem:Oseen.rot.res.exist.smooth}
there exists a solution 
$\np{\vvel_j,\vpres_j}\in\Xs(\Omega)\times\Ys(\Omega)$
to \eqref{sys:Stokes.rot.res} with $g=g_j$ for each $j\in\N$.
From
Lemma \ref{lem:Stokes.rot.res.estimates}
we infer that $\np{\vvel_j,\vpres_j}$ 
is a Cauchy sequence in 
the Banach space
$\Xs(\Omega)\times\Ys(\Omega)$.
Therefore, there exists a unique limit $\np{\vvel,\vpres}\in\Xs(\Omega)\times\Ys(\Omega)$,
which is a solution to \eqref{sys:Stokes.rot.res}.
This completes the proof.
\end{proof}

\section{The time-periodic problem}
\label{sec:tpprob}
Now we consider the time-periodic problem \eqref{sys:Stokes.rot.tp}
and prove the well-posedness results from Theorem \ref{thm:Stokes.rot.tp}.
For the proof, we reduce \eqref{sys:Stokes.rot.tp}
to the resolvent problems for each Fourier mode,
which can be solved by means of Theorem \ref{thm:Stokes.rot.res}. 
Due to the \textit{a priori} estimate \eqref{est:Stokes.rot.res},
these solutions constitute a summable sequence in a suitable space,
so that the associated Fourier series forms a solution to
the time-periodic problem \eqref{sys:Stokes.rot.tp}. 

\begin{proof}[Proof of Theorem \ref{thm:Stokes.rot.tp}]
Let $f\in\AR(\torus;\LR{q}(\Omega)^3)$.
Then there exist $f_k\in\LR{q}(\Omega)^3$, $k\in\Z$, such that
\[
f(t,x)=\sum_{k\in\Z}f_k(x)\e^{i\perf kt}.
\]
By Theorem \ref{thm:Stokes.rot.res}
there exists a solution $\np{\uvel_k,\upres_k}\in\Xk(\Omega)\times\Ys(\Omega)$
to
\begin{equation}\label{sys:Stokes.rot.res.uk}
\begin{pdeq}
i\perfs k\uvel_k+\tay\rotterm{\uvel_k} 
- \Delta\uvel_k
+ \grad \upres_k
 &= f_k
&& \tin \Omega, \\
\Div\uvel_k&=0
&& \tin\Omega, \\
\uvel_k&=0
&& \ton \partial\Omega,
\end{pdeq}
\end{equation}
which satisfies
\[
\begin{aligned}
&\norm{i\perfs k\uvel_k+\tay\rotterm{\uvel_k}}_{q}
+\norm{\grad^2 \uvel_k}_{q}
+\norm{\grad \uvel_k}_{q}
\\
&\qquad\qquad
+\norm{\grad \uvel_k}_{3q/(3-q)}
+\norm{\uvel_k}_{3q/(3-2q)}
+\norm{\upres_k}_{3q/(3-q)}
\leq C\norm{f_k}_{q}
\end{aligned}
\]
for $C$ as in Theorem \ref{thm:Stokes.rot.res}.
Since $C$ is independent of $k$,
the series
\begin{equation}\label{eq:Fseries.up}
\uvel(t,x)=\sum_{k\in\Z}\uvel_k(x)\e^{i\perf kt},
\qquad
\upres(t,x)=\sum_{k\in\Z}\upres_k(x)\e^{i\perf kt},
\end{equation}
define a pair 
$\np{\uvel,\upres}\in\XT(\torus\times\Omega)\times\YT(\torus\times\Omega)$,
which satisfies estimate \eqref{est:Stokes.rot.tp}
with the same constant $C$
and is a time-periodic solution to problem \eqref{sys:Stokes.rot.tp}.

It remains to prove the uniqueness statement.
For this purpose, consider a solution 
$\np{\uvel,\upres}\in\XT(\torus\times\Omega)\times\YT(\torus\times\Omega)$
to \eqref{sys:Stokes.rot.tp} 
with right-hand side $f=0$.
Then the Fourier coefficients 
$\np{\uvel_k,\upres_k}\in\Xk(\Omega)\times\Ys(\Omega)$, $k\in\Z$,
defined by \eqref{eq:Fseries.up},
are solutions to problem
\eqref{sys:Stokes.rot.res.uk} with $f_k=0$.
From Theorem \ref{thm:Stokes.rot.res}
we thus conclude $\np{\uvel_k,\upres_k}=\np{0,0}$
for all $k\in\Z$,
so that $\np{\uvel,\upres}=\np{0,0}$. 
This shows uniqueness of the solution and
completes the proof.
\end{proof}




 
\end{document}